\documentclass{amsart}
\usepackage{amssymb,pst-all}
\usepackage{amsbsy}
\usepackage{amscd}
\usepackage{amsmath}
\usepackage{amsthm}
\usepackage{amsxtra}
\usepackage{latexsym}

\usepackage[mathscr]{eucal}
\usepackage{upgreek}
\usepackage{wasysym}
\usepackage{xcolor}
\definecolor{rouge}{rgb}{0.85,0.1,.4}
\definecolor{bleu}{rgb}{0.1,0.2,0.9}
\definecolor{violet}{rgb}{0.7,0,0.8}

\usepackage[colorlinks=true,linkcolor=bleu,urlcolor=violet,citecolor=rouge]{hyperref} 

\newcommand{\affrho}{\widehat{\rho}}
\newcommand{\affW}{\widehat{W}}
\newcommand{\wh}{\widehat}
\newcommand{\affh}{\widehat{\mathfrak{h}}}
\newcommand{\dual}[1]{{#1}^*}
\newcommand{\inv}{^{-1}}
\newcommand{\Vg}[1]{V^{#1}(\fing)}
\newcommand{\Vs}[1]{V_{#1}(\fing)}
\newcommand{\Zhu}{A}
\newcommand{\BGG}{{\mathcal O}}
\newcommand{\bra}{{\langle}}
\newcommand{\ket}{{\rangle}}

\newcommand{\Lam}{\Lambda}
\newcommand{\germ}{\mathfrak}
\newcommand{\cprime}{$'$}
\newcommand{\on}{\operatorname}
\newcommand{\+}{\mathop{\oplus}}
\renewcommand{\*}{{\otimes}}
\newcommand{\mc}{\mathcal}
\newcommand{\mf}{\mathfrak}

\newcommand{\fing}{\mf{g}}

\newcommand{\affg}{\widehat{\mf{g}}}

\newcommand{\isomap}{{\;\stackrel{_\sim}{\to}\;}}
\newcommand{\Z}{\mathbb{Z}}
\newcommand{\C}{\mathbb{C}}
\newcommand{\N}{\mathbb{N}}
\newcommand{\Q}{\mathbb{Q}}

\newcommand{\ra}{\rightarrow}
\newcommand{\lam}{\lambda}
\newcommand{\che}{^{\vee}}
\newcommand{\affP}{\widehat{P}_k}
\newcommand{\eW}{\widetilde{W}}

\def\g{\mathfrak{g}}

\def\h{\mathfrak{h}}

\newcommand\sm{{\mathsf m}}

\newcommand\Ga{{\Gamma}}

\def\sl{\mathfrak{sl}}

\def\leq{\leqslant}
\def\geq{\geqslant}

\newlength\yStones
\newlength\xStones
\newlength\xxStones

\makeatletter
\def\Stones{\pst@object{Stones}}
\def\Stones@i#1{%
  \pst@killglue%
  \begingroup%
  \use@par%
  \setlength\xxStones{\xStones}%
  \expandafter\Stones@ii#1,,\@nil
  \endgroup
  \global\addtolength\xStones{0.6cm}%
  \global\addtolength\yStones{-7.5mm}}%
\def\Stones@ii#1,#2,#3\@nil{%
  \rput(\xxStones,\yStones){%
    \psframebox[framesep=0]{%
      \parbox[c][6mm][c]{11mm}{\makebox[11mm]{$#1$}}}}%
  \addtolength\xxStones{1.2cm}%
  \ifx\relax#2\relax\else\Stones@ii#2,#3\@nil\fi}
\makeatother

\def\Stone#1{\fbox{\makebox[11mm]{\strut#1}}\kern2pt}

\DeclareMathOperator{\gr}{gr}

\DeclareMathOperator{\ad}{ad}

\newcommand{\dcl}{\DeclareMathOperator}
\dcl\cdet{cdet} \dcl\Sp{Specm} \dcl\depth{depth} \dcl\im{Im} \dcl\Span{span} \dcl\Ker{Ker} \dcl\Specm{Specm}
\dcl\Supp{Supp} \dcl\codim{codim} \dcl\Y{Y} \dcl\gl{\mathfrak{gl}}    \dcl\U{U} \dcl\T{T}
\dcl\qdet{qdet} \dcl\sgn{sgn} \dcl\diag{diag}
\dcl\dd{{\mathrm d}}

\theoremstyle{theorem}
\newtheorem{Th}{Theorem}[section]

\newtheorem{Pro}[Th]{Proposition}

\newtheorem{Lem}[Th]{Lemma}

\newtheorem{Co}[Th]{Corollary}

\theoremstyle{remark}
\newtheorem{Def}[Th]{Definition}
\newtheorem{Rem}[Th]{Remark}

\title{Weight representations  of admissible affine vertex algebras}
\author{Tomoyuki Arakawa}
\address{Research Institute for Mathematical Sciences, Kyoto University,
 Kyoto 606-8502 JAPAN}
\email{arakawa@kurims.kyoto-u.ac.jp}

\author{Vyacheslav Futorny}
\address{Instituto de Matem\'atica e Estat\'istica, Universidade de S\~ao Paulo, S\~ao Paulo-SP,
Brazil}
\email{futorny@ime.usp.br}
\author{Luis Enrique Ramirez}
\address{Universidade Federal do ABC, Santo Andr\'e-SP,
Brazil}
\email{luis.enrique@ufabc.edu.br}
\begin{document}

\begin{abstract}
For an admissible affine vertex algebra 
$\Vs{k}$
of type $A$, 
we describe a new family of relaxed highest weight representations of $\Vs{k}$.
They are simple quotients of representations of the affine Kac-Moody algebra $\affg$ induced from
the following $\fing$-modules:
1) generic Gelfand-Tsetlin modules in the principal nilpotent orbit,
 in particular  all such modules induced from $\mf{sl}_2$;
2)
all Gelfand-Tsetlin modules 
  in the principal nilpotent orbit which are  induced from  $\mf{sl}_3$; 3) 
  all  simple Gelfand-Tsetlin modules over  $\mf{sl}_3$.
  This in particular gives the classification 
  of all simple  positive energy weight representations of $\Vs{k}$
  with finite dimensional weight spaces for $\fing=\mf{sl}_3$.

\end{abstract}

\maketitle

\section{Introduction}
Let $\fing$ be a complex simple finite dimensional Lie algebra,
$\affg=\fing[t,t\inv]+\C K$ the non-twisted affine Kac-Moody algebra associated with
$\fing$,
$\Vg{k}$   the universal affine vertex algebra
associated with $\fing$ at a non-critical level $k$,
$\Vs{k}$ the unique simple quotient of $\Vg{k}$.

The simple affine vertex algebra
$\Vs{k}$ 
is called {\em admissible} if it is isomorphic
to  an admissible representation \cite{KacWak89}
as  a $\affg$-module.
In the case when  $\fing=\mf{sl}_2$,
the modular properties and the Verlinde formula for $\Vs{k}$-modules
 were studied in  
\cite{Rid09,Rid10,CreRid12,CreRid13}.
To generalize these results to an arbitrary $\fing$,
one needs to classify  
 simple positive energy weight representations of $\Vs{k}$,
that is, 
$\Vs{k}$-modules that appear as
simple quotients of $\affg$-modules induced from
 simple weight representations of $\fing$ 
which are not necessarily highest weight. 

\medskip

By a well-known result of Zhu \cite{Zhu96},
there is a one-to-one correspondence
between
 simple positive energy representations of 
a graded  vertex algebra $V$
and simple $\Zhu(V)$-modules, where
$\Zhu(V)$ is {\em Zhu's algebra} of $V$.
In the case when
$V$ is an affine vertex algebra
$\Vs{k}$,
we have
\begin{align}
 \Zhu(\Vs{k})\cong U(\fing)/I_k
\label{eq:zhu-of-simple}
\end{align}
for some
two sided-ideal $I_k$ of the universal enveloping algebra
$U(\fing)$ of $\fing$.
Thus,
  a simple $U(\fing)$-module $M$
is an $\Zhu(\Vs{k})$-module if and only if  the annihilating ideal of $M$
in $U(\fing)$
contains  $I_k$.

In \cite{A12-2} the first named author  classified
primitive ideals of $U(\fing)$ containing $I_k$
for an admissible affine
vertex algebra $\Vs{k}$, proving the Adamovic-Milas conjecture \cite{AdaMil95}.
Such primitive ideals are exactly the annihilators
of  simple highest weight  representations
$L(\lam)$ of $\fing$
that appear as the top weight  part of a level $k$ admissible modules
$\widehat{L}_k(\lam)$
of $\affg$ 
that are $\Vs{k}$-modules.
Hence our problem amounts to classifying simple weight modules of
$U(\fing)$
having those annihilating ideals.

This is a  difficult problem. Complete classification
of simple weight  modules of $\fing$ is only known 
  for $\fing=\mathfrak{sl}_2$. It remains open for all other simple Lie algebras
although certain categories of modules were studied extensively (finite dimensional modules, category $\mathcal O$ etc.)
In particular, simple weight 
 modules which have finite dimensional weight spaces 
have been classified by Fernando \cite{Fe90} and Mathieu \cite{Mat00}. Among weight modules with infinite dimensional weight spaces the most studied is the class of Gelfand-Tsetlin modules for type $A$. 
On such modules the Gelfand-Tsetlin subalgebra (being a maximal commutative subalgebra) of the universal enveloping algebra of $\mf{sl}_n$ has a generalized eigenspace decomposition. Gelfand-Tsetlin  modules were introduced in
\cite{DFO89}, \cite{DFO92}, \cite{DFO94}. Gelfand-Tsetlin subalgebras are related to  important problems in representation theory of Lie algebras \cite{FO10}, \cite{FO14}, general hypergeometric functions on the complex Lie group $GL(n)$, \cite{Gr04}; problems in classical mechanics,  \cite{KW06}, \cite{KW6} among the others.

In this paper we describe several new families of simple positive energy weight representations of admissible affine vertex algebra
$\Vs{k}$
 for $\fing=\mf{sl}_n$.  More precisely, we construct explicitly the following classes of   simple weight $\Zhu(\Vs{k})$-modules 
  for $\fing=\mf{sl}_n$:
 
 \noindent - generic Gelfand-Tsetlin $\mf{sl}_n$-modules
in the principal nilpotent orbit
(Corollary \ref{cor-generic}, Theorem \ref{Th: Admissible induced by sl(2)}).
Here we say a simple $\fing$-module is in a nilpotent orbit $\mathbb{O}$ if 
the associated variety of its annihilating ideal is the closure of $\mathbb{O}$.
In this class of representations
those with finite dimensional weight spaces  are either highest weight modules or induced from an $\mathfrak{sl}_2$-subalgebra.
 Otherwise they have  infinite-dimensional weight spaces;\\
 - $\mf{sl}_n$-modules induced from simple  Gelfand-Tsetlin  $\mf{sl}_3$-modules
 in the principal nilpotent orbit
  (Theorem \ref{Th: Admissible induced by sl(3)}).
 All these module have infinite dimensional weight spaces;\\
 - Gelfand-Tsetlin $\mf{sl}_3$-modules (Theorem \ref{Th: Modules associated with minimal orbit} and Theorem \ref{thm-q>3}).
 We give a  complete classification of Gelfand-Tsetlin  $\Zhu(\Vs{k})$-module for  $\fing=\mf{sl}_3$.
  This includes all simple weight modules with finite weight multiplicities (both cuspidal and induced) as well as series of modules with infinite weight multiplicities.  \\
 
 
Hence 
 we obtain  in particular a  complete classification of all simple relaxed highest weight representations  with finite dimensional
 weight spaces of admissible affine vertex algebras
 $V_k(\fing)$
 for $\fing=\mf{sl}_3$.
   We note that a family of weight modules  for $V_k(\mf{sl}_3)$  has been  previously constructed in  
\cite{Ad16}.
\\

Our approach is based on the theory of Gelfand-Tsetlin modules developed in \cite{FGR}, \cite{FGR15}, \cite{FGR16}. In these papers explicit basis was constructed for different classes of simple Gelfand-Tsetlin modules for $\mathfrak{gl}_n$, including 
generic and $1$-singular cases. This basis generalizes the classical Gelfand-Tsetlin basis for finite dimensional representations.   In \cite{FGR} a complete classification and explicit construction of all simple Gelfand-Tsetlin modules for $\mathfrak{sl}_3$
is given. In all other cases the problem of classifying the simple Gelfand-Tsetlin modules is still open. \\

{\bf{Acknowledgments}}. 
The authors thank the anonymous referees for careful reading and helpfull comments.
T.A.  is supported in part by the Fapesp grant
(2015/06469-6)  and by JSPS KAKENHI Grant Numbers 25287004 and 26610006.
 He gratefully acknowledges the hospitality and excellent
working conditions at the S\~ao Paulo University where this work was done.
V.F. is
supported in part by the CNPq grant (301320/2013-6) and by the
Fapesp grant (2014/09310-5).

\section{Weight modules and admissible representations}
\subsection{Weight modules}
Let $\fing$ be a complex simple finite dimensional Lie algebra with a fixed 
triangular decomposition $\fing=\mf{n}_-\+\h\+\mf{n}_+$.
A $\fing$-module $M$ is called  \emph{weight} if $\h$ is diagonalizable on $M$. For $\lambda\in \h^*$ the subspace $M_{\lambda}$ of those $v\in V$ such that $hv=\lambda(h)v$ is the \emph{weight subspace} of weight $\lambda$.
 The set of all those $\lambda\in \h^*$ for which $M_{\lambda}\neq 0$ is the \emph{weight support} of $M$ and the dimension of  $M_{\lambda}$ is the \emph{weight multiplicity} of $\lambda$.

Let $\mathfrak p=
\mathfrak l\oplus \mathfrak m$ be a  parabolic subalgebra of $\fing$
with the Levi subalgebra $\mathfrak l$.
 If $N$ is a weight module over  $[\mf{l},\mf{l}]$ then one defines on it a structure of a  $\mathfrak l$-module 
by choosing any $\lambda\in (\h^{\perp})^*$ and setting $hv=\lambda(h)v$ for any $h\in \h^{\perp}$ and any $v\in N$. Here, $\h^{\perp}$ is the orthogonal complement of $[\mf{l},\mf{l}]\cap\h$ 
with respect to the Killing form.  Furthermore, consider 
$N$ as a $\mathfrak p$-module with a trivial action of
$\mathfrak m$. 

Now construct the induced $\fing$-module $M_{\mathfrak p}(\lambda, N)
=U(\fing)\otimes_{U(\mathfrak p)} N$.   
 If $N$ is simple $[\mf{l},\mf{l}]$-module then $M_{\mathfrak p}(\lambda, N)$ has a unique simple quotient $L_{\mathfrak p}(\lambda, N)$. 

When $\mf{p}$ is the fixed Borel subalgebra $\h+\mf{n}_+$,
we write $M(\lam)$ for $M_{\mathfrak p}(\lambda, \C)$
and $L(\lam)$ for $L_{\mathfrak p}(\lambda, \C)$.

We say that a simple weight $\fing$-module $L$ is \emph{cuspidal}  if it is not isomorphic to $L_{\mathfrak p}(\lambda, N)$ for any choice of a parabolic subalgebra ${\mathfrak p}\neq \fing$ and any choice of $N$. By \cite{DMP00}, if $L$ is cuspidal then  its weight support coincides with a coset of some  weight (any weight from the support) by the root lattice.  

\subsection{Affine vertex algebras}
\label{subsec:relaxed}
Let \begin{align}
\affg=\fing[t,t\inv]\+ \C K
\end{align}
be the affine Kac-Moody algebra
associated with $\fing$.
The commutation relation of $\affg$
is given by
\begin{align}
& [xt^m,yt^n]=[x,y]t^{m+n}+ m(x|y)\delta_{m+n,0}K,
\quad [K,\affg]=0,
\end{align}
where $(~|~)$ is the normalized invariant bilinear form of $\fing$.

For $k\in \C$,
let 
$\Vg{k}$
be the universal affine vertex algebra associated with $\fing$
at level $k$.
By definition
\begin{align}
 \Vg{k}=U(\affg)\*_{U(\fing[t]\+ \C K)}\C_k,
\end{align}
where
 $\C_k$ is the one-dimensional
representation of 
$\fing[t]\+ \C K$ on which 
$\fing[t]$ acts trivially and $K$ acts as a  multiplication by $k$.
The space $\Vg{k}$ is 
equipped with the natural vertex algebra structure (see
 \cite{Kac98,FreBen04}),
 and it is conformal
 by the Sugawara construction provided that $k\ne -h^{\vee}$,
 where $h^{\vee}$ is the dual Coxeter number of $\fing$.
 The unique simple graded quotient  $\Vs{k}$ of $\Vg{k}$ is called
 the {\em simple affine vertex algebra} associated with $\fing$
 at level $k$.
 
 For a conformal vertex algebra $V$,
 let $\Zhu(V)$ be Zhu's algebra of $V$.
 For a positive energy representation $M$ of $V$,
 the top conformal weight space $M_{top}$
 is naturally a module over $\Zhu(V)$,
 and moreover,
 the correspondence
 $M\mapsto M_{top}$ gives a one-to-one correspondence
 between
 simple positive energy representation of $V$ and
 simple $V$-modules (\cite{Zhu96}).

 In the case that $V=\Vg{k}$,
 we have the natural isomorphism
 \begin{align}
  \Zhu(\Vg{k})\cong U(\fing),
 \end{align}
 and hence,
 \begin{align}
  \Zhu(\Vs{k})\cong U(\fing)/I_k
 \end{align}
 for some two-sided ideal $I_k$ of $U(\fing)$.

A $\Vs{k}$-module is called {\em weight} it is a direct sum of weight modules over $\fing$.
In this paper we study  the positive energy weight representations
of an admissible affine vertex algebra $\Vs{k}$.

For a simple weight module  $E$ of $\fing$, 
let $\mathbf{M}_k(E)=U(\affg)\*_{U(\fing[t]\+ \C K)}E$,
where $E$ is considered as a $\fing[t]\+ \C K$-module on which $\fing[t]t$ acts trivially and $K$ acts the multiplication by $k$.
It is clear that $\mathbf{M}_k(E)$ is a weight module and so is its simple quotient $\mathbf{L}_k(E)$.

According to Zhu's Theorem  above
the correspondence
$E\mapsto \mathbf{L}_k(E)$
gives the  one to one correspondences between  simple weight representations of $A(\Vs{k})=U(\fing)/I_k$  and 
 simple relaxed positive energy highest weight representations  of $\Vs{k}$.

A positive energy weight representation is an example of a {\em relaxed highest representation} \cite{FST} of $\affg$.
A general simple relaxed highest representation is obtained 
from a simple positive energy weight representation by applying a spectral flow.
 \subsection{Admissible representations of $\affg$}
Let $\Delta$ be
the set  of roots,
$\Delta_+$ the set  of positive roots,
 $W$ the Weyl group of $\fing$.
Put $Q=\sum_{\alpha\in \Delta}\Z \alpha$,
$Q^\vee=\sum_{\alpha\in \Delta}\Z \alpha^{\vee}$,
where $\alpha^{\vee}=2\alpha/(\alpha,\alpha)$.
Let $P$ be the  weight lattice of $\fing$,
$P^{\vee}$ the coweight lattice.
Denote by $\theta$ and $\theta_s$ the highest root and the highest short root of $\fing$,
respectively.
We have $(\theta,\theta)=2$ and $(\theta_s,\theta_s)=2/r^{\vee}$,
where $r^{\vee}$ is the lacing number of $\fing$,
that is, 
the maximum number of the edges in the Dynkin diagram of $\fing$.
Let $\rho$ be the half sum of positive roots,
$\rho^{\vee}$  the half sum of positive coroots.
  
We fix the  Cartan subalgebra $\widehat{\h}$ of $\affg$
as
 $\widehat{\h}=\h\+ \C K$.  We will extend $\affg$ by a derivation $D$ and denote the extended algebra by $\widetilde{\g}$:
 $\widetilde{\g}=\fing[t,t\inv]\+ \C K + \C D$. 
 Let  $\widetilde{\h}=\h\+ \C K\+ \C D$ be the extended Cartan subalgebra of $\widetilde{\g}$,
 $\widetilde{\h}^*=\h^*\+ \C \Lam_0\+ \C \delta$ its dual,
 where $\Lam_0(K)=\delta(D)=1$,
 $\Lam(\h+\C D)=\delta(\h\+ C K)=0$.
 The dual $\affh^*$ of $\affh$ is identified with the subspace $\h\+\C \Lam_0\subset \widetilde{\h}^*$.
 
 Let $\widehat{\Delta}$ be the set of roots of $\affg$ in $\widetilde{\h}^*$,
  $\widehat{\Delta}^{re}\subset \widehat{\Delta}$ the set of real root of $\affg$,
   $\widehat{\Delta}^{re}_+$ the set of real positive roots.
   We have
   $\widehat{\Delta}^{re}=\{\alpha+n\delta\mid \alpha\in \Delta, n\in \Z\}$,
   $\widehat{\Delta}^{re}_+=\{\alpha+n\delta\mid \alpha\in \Delta_+, n\in \Z_{\geq 0}\}\sqcup
   \{-\alpha+n\delta\mid \alpha\in \Delta_+,n\in \Z_{\geq 1}\}$.

   Denote by $\widehat{W}$ the  Weyl group of $\affg$.
Then  $\affW=W\ltimes Q^{\vee}$.   
Let $\eW=W\ltimes P^{\vee}$, the extended affine Weyl group of $\fing$.
For $\mu\in P^{\vee}$,
we denote by $t_{\mu}$ the corresponding element of $\eW$.
We have
$\eW= \widetilde{W}_+\ltimes \affW$,
where $\widetilde{W}_+$ is a finite subgroup of $\eW$ consisting of elements of length zero.
The group $\widetilde{W}_+$  is described as follows.
Write
$\theta$ as a sum of simple roots: $\theta=\sum_{i=1}^la_i \alpha_i$,
and set
$J=\{i\in \{1,\dots, l\};
a_i=1\}$.
Then we
have
\begin{align}
 \eW_+=\{t_{\varpi_j}w_j; j\in J\},
\label{eq:elements-of-eW-of-Dynkin-auto}
\end{align}
where 
$\varpi_j$ is the $j$-th fundamental weight of $\fing$
and $w_j$ is the unique element of $W$
which fixes the set $\{\alpha_1,\dots,\alpha_{l},-\theta\}$
and $w_j(-\theta)=\alpha_j$.

 For $\lam\in \dual{\affh}$,
let $\wh\Delta(\lam)$
and $\affW(\lam)$ be its integral root system
and its integral Weyl group,
respectively,
that is,
 \begin{align}
& \wh\Delta(\lam)
=\{\alpha\in \wh{\Delta}^{re}; \bra
 \lam+\affrho,\alpha\che\ket
\in \Z\},
&\affW(\lam)=\bra s_{\alpha}; \alpha\in\wh \Delta(\lam) \ket\subset \affW,
\end{align}
where $s_{\alpha}$ is the reflection corresponding to $\alpha$
and  $\affrho=\rho+h^{\vee}\Lam_0$.
Let
$\wh\Delta(\lam)_+=\wh{\Delta}(\lam)\cap \wh{\Delta}^{re}_+$,
the set of positive roots of 
$\wh{\Delta}(\lam)$
and 
$\wh\Pi(\lam)\subset \wh{\Delta}(\lam)_+$, the set of simple roots.

For $\lam\in \dual{\h}$ and $k\in \C$,
set $\widehat{L}_k( \lam):=\mathbf{L}_k(L(\lam))$,
which  is the irreducible highest weight representation of $\affg$ with highest weight 
$\widehat{\lam}:=\lam+k\Lam_0\in \affh^*$.

Recall that 
a weight $\lam\in \dual{\affh}$ is called {\em admissible} \cite{KacWak89} if
\begin{enumerate}
 \item $\lam$ is regular dominant, that is,
$\bra \lam+\affrho,\alpha\che\ket \not \in \{0,-1,-2,\dots\}$ for all $\alpha\in \widehat{\Delta}^{re}_+$;
 \item $\Q\wh\Delta(\lam)=\Q \wh{\Delta}^{re}$.
\end{enumerate}
The simple highest weight representation $\widehat{L}_k(\lam)$  called admissible if
$\widehat{\lam}$ is admissible.

A complex number $k$ is called 
  admissible (for $\affg$) 
  if the affine vertex algebra
  $\Vs{k}$ is admissible as an $\affg$-module.
  (Note that $\Vs{k}\cong \widehat{L}_k(0)$).
\begin{Pro}[\cite{KacWak89,KacWak08}]
\label{Pro:admissible number}
The number  $k$ is admissible if and only if 
 \begin{align*}
  k+h\che=\frac{p}{q}
 \quad \text{with }p,q\in \N,\ (p,q)=1,\ p\geq 
\begin{cases}
h\che&\text{if }(r\che, q)=1\\
h&\text{if }(r\che,q)=r\che,
\end{cases}
 \end{align*}
where 
$h$  is
 the Coxeter number of $\fing$.
If this is the case
we have
 $\wh{\Pi}(k\Lam_0)=\{\dot{\alpha_0},\alpha_1,\alpha_2,\dots,\alpha_l\}$,
where 
\begin{align*}
  \dot{\alpha_0}= 
\begin{cases}
-\theta+q\delta&\text{if }(r\che,q)=1
\\
 -\theta_s+\frac{q}{r\che}\delta&\text{if }(r\che ,q)=r\che.
\end{cases}
\end{align*}
\end{Pro}

For an admissible number $k$,
let $Pr_k$ be the set of admissible weights
$\lam$ such that
there exists $y\in \eW$ satisfying
 $\wh{\Delta}(\lam)=y( \wh{\Delta}(k\Lam_0))$.
Set
\begin{align}
 Pr_{k,\Z}=Pr_k\cap \affP,
\end{align}
where
\begin{align}
 \affP=\{\lam\in \dual{\affh};
\lam(K)=k,\ \bra \lam,\alpha_i\che\ket \in \Z
\ \text{for all }i=1,\dots,l \}.
\label{eq:affP}
\end{align}
Then
\begin{align}
 Pr_{k,\Z}=
\begin{cases}
 \{\lam\in \affP;
\bra \lam,\alpha_i\che \ket \geq 0
\ \text{for $i=1,\dots, l$,\ }\bra \lam,\theta\ket\leq
 p-h\che\}&\text{if }
(r\che
 ,q)=1,\\
\{\lam\in \affP; \bra \lam,\alpha_i\che \ket \geq 0
\ \text{for $i=1,\dots, l$,\ }\bra \lam,\theta_s\che \ket \leq
 p-h\}&\text{if }
(r\che ,q)=r\che.\end{cases}
\label{eq:Pr+}
\end{align}
Note that
\begin{align}
 Pr_{k,\Z}\cong \begin{cases}
	      \widehat{P}_{+}^{p-h\che}&\text{if }(r\che,q)=1,\\
{}^L\widehat{P}_{+}^{\vee,p-h}&\text{if }(r\che,q)=r\che,
	     \end{cases}
\end{align}
where $\widehat{P}_{+}^m$
is the set of  level $m$  integral dominant weights of $\affg$,
and ${}^L\widehat{P}_{+}^{\vee,m}$ is the set of level $m$ integral dominant
coweights
of the affine Kac-Moody  algebra $\widehat{{}^L\fing}$
associated with the Langlands dual Lie algebra ${}^L\fing$.

For 
$\lam\in Pr_{k,\Z}$
we have \begin{align}\label{eq:integral-roots}
 \widehat{\Delta}(\lam)
=\begin{cases}
			 \{\alpha+nq\delta; \alpha\in \Delta,\ n\in
			 \Z\}&\text{if }(q,r\che)=1,\\
\{\alpha+nq\delta;\alpha\in \Delta_{long}\}\sqcup
			 \{\alpha+\frac{nq}{r\che}\delta; \alpha\in\
			 \Delta_{short}, n\in \Z\}
&\text{if }(q,r\che)=r\che,
			\end{cases}
\end{align}
where $\Delta_{long}$ (resp.\ $\Delta_{short}$) is  the set  of
long roots
(resp.\ short roots) of $\fing$.
It follows that
\begin{align}
 \widehat{W}(\lam)=\begin{cases}
					 W\ltimes q Q\che&\text{if
					 }(q,r\che)=1,\\
W\ltimes q Q&\text{if }(q,r\che)=r\che
					\end{cases}
\end{align}
for $\lam\in Pr_{k,\Z}$.
In particular
\begin{align}
 \widehat{W}(\lam)\cong \begin{cases}
					 \widehat{W}&\text{if
					 }(q,r\che)=1,\\
{}^L \widehat{W}&\text{if }(q,r\che)=r\che,
					\end{cases}
\end{align}
for $\lam\in Pr_k$,
where ${}^L\widehat{W}$ is the Weyl group of $\widehat{{}^L\fing}$.

If $k$ is an admissible number with denominator $q$ such that
 $(q,r^{\vee})=1$,
we have  \cite{KacWak89}
\begin{align}
 Pr_k=\bigcup_{y\in \eW
\atop y(\wh{\Delta}(k\Lam_0)_+)\subset 
\wh{\Delta}^{re}_+}Pr_{k,y},
\quad Pr_{k,y}:=y\circ Pr_{k,\Z},
\label{eq:KW-description-of-admissible-weights}
\end{align}
Moreover,
\begin{align}
Pr_{k,y}\cap Pr_{k,y'}\ne \emptyset
\iff Pr_{k,y}= Pr_{k,y'}\iff y'=yt_{q\varpi_j}w_j
\label{eq:Pr=Pr}
\end{align}
with some $j\in J$
for $y,y'\in \eW$ such that 
$y(\wh{\Delta}(k\Lam_0)_+),
y'(\wh{\Delta}(k\Lam_0)_+)\subset 
\wh{\Delta}^{re}_+$.

 \subsection{Admissible representations and nilpotent orbits}
 \label{subsection:Admissible representations and nilpotent orbits}
 Let $k$ be  an admissible number for $\affg$.

 We extend the concept of admissibility to 
 modules over $\fing$. Namely, we say
 that a $\fing$-module $M$
is {\emph{admissible}} of level $k$ if 
 $\mathbf{L}(M)$ is an $\Vs{k}$-module,
 or equivalently,
 $M$ is an $A(\Vs{k})$-module.
 Thus,
$M\mapsto \mathbf{L}(M)$ 
gives a one-to-one correspondence
between 
the set of the isomorphism classes of admissible $\fing$-modules of level $k$
and that of simple positive energy representations of $\Vs{k}$.

Let 
 $$\overline{Pr_k}=\{\bar \lam\mid \lam\in Pr_k\}\subset \h^*,$$
where 
$\bar \lam$ is the projection of $\lam$ to $\h^*$.

  \begin{Th}[\cite{A12-2}]\label{Th:classification-of-admissible-modules}
  Let $k$ be admissible,
$\lam\in \h^*$.
Then $\widehat{L}_k(\lam)$ is a module over $\Vs{k}$ if and only if
  $\lam \in \overline{Pr_k}$.
  Equivalently, 
   $L(\lam)$ is a
  $\Zhu(\Vs{k})$-module
  if and only if  $ \lam\in \overline{Pr_k}$.
      \end{Th}
Theorem \ref{Th:classification-of-admissible-modules} in particular says that
any $\Zhu(\Vs{k})$-module in  the category $\BGG$ is completely reducible since
all elements of  $\overline{Pr_k}$ are regular dominant, that is,
$\bra \lam+\rho,\alpha^{\vee}\ket\not \in \{0,-1,-2,-3,\dots\}$ for all $\alpha\in \Delta_+$.

For $\lam\in \h^*$,
 let $$J_{\lam}=\on{Ann}_{U(\g)}L(\lam).$$ 
 By Duflo's Theorem \cite{Duf77}, one knows that
for
any primitive ideal $I$
(that is, the annihilator of some simple module) of $U(\fing)$ there exists $\lam\in \h^*$ such that
$I=J_{\lam}$.

Thus, we obtain following assertion immediately from Theorem \ref{Th:classification-of-admissible-modules}.
\begin{Co}
Let $k$ be an admissible number,
 $M$ a simple $U(\fing)$-module.
 Then $M$ is an admissible $\g$-module of level $k$ if and only if 
$\on{Ann}_{U(\fing)}M=J_{ \lam}$ for some 
$ \lam\in \overline{Pr_k}$.
\end{Co}

 Let $\mc{Z}(\fing)$ be the center of $U(\fing)$,
 and let $\chi_{\lam}:\mc{Z}(\fing)\ra \C$ be the evaluation at $L(\lam)$.
By \cite{Joseph:1979kq},
one knows that
 the correspondence $I\mapsto IM(\lam)$ 
gives an order-preserving injection between the set of two sided ideals 
of $U(\fing)$
containing 
$U(\fing)\ker \chi_{ \lam}$
 and submodules of $M(\lam)$,
 and in fact, this map is a bijection if $\lam$ is regular.

 \begin{Pro}\label{Pro:equiv-class-adm}
   For $\lam \in \overline{Pr_k}$,
  the primitive ideal $J_{\lam}$ is the unique maximal two-sided ideal of $U(\fing)$ containing 
  $U(\fing)\ker \chi_{\lam}$. In particular, 
for $\lam,\mu \in \overline{Pr_k}$,
  $J_{ \lam}=J_{ \mu}$ if and only if
  there exists $w\in W$ such that
  $ \mu=w\circ \lam$.
      \end{Pro}
 \begin{proof}
Let $ \lam\in \overline{Pr_k}$.
Then $I_k\subset J_{\lam}$ 
by Theorem \ref{Th:classification-of-admissible-modules}.
(Recall that  $A(L_{k}(\fing))=U(\fing)/I_k$). Thus,
 $M( \lam)/J_{ \lam}$  is a highest weight
$A(L_{k}(\fing))$-module,
and therefore,
it must be isomorphic to $L( \lam)$ by Theorem \ref{Th:classification-of-admissible-modules}.
This means that
 $J_{ \lam}M( \lam)$ is the unique  maximal submodule of 
$M( \lam)$,
and therefore,
 $J_{\lam}$ is the unique maximal two-sided ideal of $U(\fing)$ containing 
  $U(\fing)\ker \chi_{ \lam}$. 
 \end{proof}

 For a two-sided ideal $I$ of $U(\fing)$,
let $\on{Var}(I)$ be the {\em associated  variety} of $I$,
that is,
the zero locus of 
 $\gr I$ in $\fing^*$,
 where 
 $\gr I$ is the associated graded with respect to the filtration induced from the PBW filtration of $U(\fing)$.
 By Joseph's Theorem \cite{Jos85},
 for a primitive ideal $I$ 
we have
$\on{Var}(I)=\overline{\mathbb{O}}$  for some nilpotent orbit $\mathbb{O}$ of $\fing$.

We say a simple $\fing$-module $M$ is {\em in the orbit $\mathbb{O}$}  if
$\on{Var}(\on{Ann}_{U(\fing)}M)=\overline{\mathbb{O}}$.
 
 \begin{Th}[\cite{A2012Dec}]\label{Th;subsquence}
  Let $k$ be an admissible number for $\affg$ with denominator $q\in \N$.
  There exists
  a nilpotent orbit
  $\mathbb{O}_q$ that depends only on $q$  such that
  \begin{align*}
   \on{Var}(I_k)=\overline{\mathbb{O}_q}.
  \end{align*}
  Explicitly,
we have
\begin{align*}
\overline{\mathbb{O}}_q
= \begin{cases}
   \{x\in \fing; (\ad x)^{2q}=0\}&\text{if }(r\che, q)=1,\\
\{x\in \fing; \pi_{\theta_s}(x)^{2q/r\che}=0\} &\text{if }(r\che,q)=r\che,
\end{cases}\end{align*}
where $r\che$ is the lacing number of $\fing$,
$\theta_s$ is the highest short root of $\fing$
and $\pi_{\theta_s}$ is the simple finite-dimensional
representation of $\fing$
with highest weight $\theta_s$.
 \end{Th}
See \cite{Ara09b} for a more concrete description of the orbits $\mathbb{O}_q$.

 Let $q\in \N$ be the denominator $q$
 of $k$.
For a nilpotent orbit $\mathbb{O}$ of $\fing$,
 set
 \begin{align}
  \overline{Pr}_k^{\mathbb{O}}=\{\lam\in \overline{Pr}_k\mid \on{Var}(J_{ \lam})=\overline{\mathbb{O}}\}.
 \end{align}
 For $\lam\in \overline{Pr}_k$
 we have 
 $\on{Var}(J_{\lam})\subset \on{Var}(I_k)=\overline{\mathbb{O}_q}$
since $I_k\subset J_{\lam}$.
Hence
\begin{align}
 \overline{Pr}_k=\bigsqcup_{\mathbb{O}\subset \overline{\mathbb{O}_q}}\overline{Pr}_k^{\mathbb{O}}.
\end{align}
Note that
\begin{align}
 \overline{Pr}_k^{\{0\}}=\overline{Pr}_{k,\Z}:=\{\bar \lam\mid \lam\in Pr_{k,\Z}\}.
\end{align}

We define an equivalence relation in  
$\overline{Pr}_k$ by 
\begin{align}
\text{$\lam\sim \mu$ $\iff$ there exists $w\in W$ such that $\mu=w\circ \lam$.
}
\end{align}
 Set
   \begin{align}
[\overline{Pr}_k]=\overline{Pr}_k/\sim.
  \end{align}
   By Proposition \ref{Pro:equiv-class-adm},
 $J_{\lam}$ depends only on the class of $\lam\in {\overline{Pr}_k}$ in $[\overline{Pr}_k]$.
 We have
 \begin{align}
  [\overline{Pr}_k]=\bigsqcup_{\mathbb{O}\subset \overline{\mathbb{O}_q}}[\overline{Pr}_k^{\mathbb{O}}],
 \end{align}
 where $[\overline{Pr}_k^{\mathbb{O}}]$ is the image of $\overline{Pr}_k^{\mathbb{O}}$ in $  [\overline{Pr}^k]$.
 
 We conclude that
 a simple $U(\fing)$-module $M$ in the orbit $\mathbb{O}$ is an admissible $\fing$-module of level $k$
 if and only if $\on{Ann}_{U(\fing)}M=J_{\lam}$ for some $\lam\in [\overline{Pr}_k^{\mathbb{O}}]$.
 

  \subsection{Principal nilpotent orbit}
  Let $k$ be an admissible number with denominator $q\in \N$,
  and let
 $\mathbb{O}_{prin}$ be
the principal nilpotent orbit,  that is, the unique dense  orbit in the nilpotent  cone $\mc{N}$ of $\fing^*$.
We have
$\on{Var}(J_{\lam})=\overline{\mathbb{O}_{prin}}=\mc{N}$
if and only if $$\on{Ann}_{U(\fing)}{L( \lam)}=U(\fing)\ker \chi_{ \lam}.$$
Let $\lam \in \overline{Pr}_k$.
Since admissible weights are regular dominant,
 $\lam \in \overline{Pr}_k^{\mathbb{O}_{prin}}$ 
 if and only if 
 $\widehat{\lam}=\lam+k\Lam_0$ is  {\em non-degenerate} in the sense of \cite{FKW92},
that is,
$\widehat{\Delta}(\lam)\cap \Delta=\emptyset$,
so that 
$L( \lam)=M( \lam)$.

Clearly, the set $ \overline{Pr}_k^{\mathbb{O}_{prin}}$ is non-empty 
if and only if 
$\mathbb{O}_q=\mathbb{O}_{prin}$.
By Theorem \ref{Th;subsquence}
this holds if and only if 
$$q\geq \begin{cases}
h&\text{if $(r^{\vee},q)=1$},\\
r^{\vee}{}^L h^{\vee}&\text{if $(r^{\vee},q)=r^\vee$},
\end{cases}$$
where $h$ is the  Coxeter number of $\g$ and ${}^L h^{\vee}$ is the dual Coxeter number of 
the Langlands dual ${}^L\fing$ of $\fing$.
\begin{Pro}(\cite{FKW92})\label{Pro:principal}
 Let $k$ be an admissible number for $\affg$, $q\in \N$ the denominator of $k$.
Suppose that $(r^{\vee},q)=1$ and $q\geq h$.
We have a bijection
\begin{align*}
\begin{array}{ccc}
(\widehat{P}^{p-h^{\vee}}_+ \times \widehat{P}_+^{\vee,q-h})/\eW_+ & \isomap & [\overline{Pr}_k^{\mathbb{O}_{prin}}],\\
\quad [(\lam,\mu)]& \mapsto  &[\bar \lam-\frac{p}{q}(\bar \mu+\rho^{\vee} ) ],
\end{array}
\end{align*}
where 
$\eW_+$ acts diagonally on $\widehat{P}^{p-h^{\vee}}_+ \times {}^L\widehat{P}_+^{q-h}$.

\end{Pro}
There is a similar description 
for the cases that $(q,r^{\vee})=r^{\vee}$,
but we omit it since we will not need it in this paper.

\subsection{Minimal nilpotent orbit in type $A$}

Let $\fing=\mf{sl}_n$.
\begin{Lem}\label{Lem:theta}
Let  $\bar y\in W$.
If $\bar y\ne 1$, 
then there exists  $j\in J$ such that
 $\bar y \bar \pi_j(\theta)\in \Delta_+$.
\end{Lem}
\begin{proof}
We have $\bar y \bar \pi_j(\theta)=-\bar y(\alpha_j)$.
Since $J=\{1,2,\dots, n-1\}$ for $\fing=\mf{sl}_n$ and
$\bar y\ne 1$,
there exists $j\in J$ such that $\bar y(\alpha_j)\in \Delta_-$.
\end{proof}

Recall the decomposition \eqref{eq:KW-description-of-admissible-weights}.
Set $\overline{Pr}_{k,y}=\{\bar \lam\mid \lam\in Pr_{k,y}\}$
and denote by  $[\overline{Pr}_{k,y}]$ the image of $\overline{Pr}_{k,y}$ in $[ \overline{Pr}_{k}]$.

\begin{Pro}
Let $k$ be an admissible number for $\affg$,
$q\in \N$ the denominator of $k$.
We have
\begin{align*}
[\overline{Pr}_k]=\bigcup_{\eta\in P_+^{\vee}\atop \theta(\eta)\leq q-1}[\overline{Pr}_{k,{t_{-\eta}}}],
\end{align*}
where $P_+^{\vee}$ is the set of dominate coweights.
\end{Pro}
\begin{proof}
Let  $y=\bar yt_{-\eta}\in \eW$ with $y\in W$, $\eta\in P^{\vee}$.
It is straightforward to see that the condition $y(\widehat{\Delta}(k\Lam_0))\subset \widehat{\Delta}_+^{re}$
is equivalent  to that
\begin{align}
\begin{cases}
0\leq \alpha(\eta)\leq q-1&\text{if }\bar y(\alpha)\in \Delta_+,\\
1\leq \alpha(\eta)\leq  q &\text{if }\bar y(\alpha)\in \Delta_-,
\end{cases}
\end{align}
for all $\alpha\in \Delta_+$.

In particular,
for $\eta\in P^{\vee}$,
 $t_{-\eta}(\widehat{\Delta}(k\Lam_0))\subset \widehat{\Delta}_+^{re}$
if and only if $\eta\in P_+^{\vee}$ and $ \theta(\eta)\leq q-1$. 
This together with \eqref{eq:KW-description-of-admissible-weights}
 shows the inclusion  $\supset $.
Conversely,
suppose that 
$y=\bar yt_{-\eta}\in \eW$ with $y\in W$, $\eta\in P^{\vee}$
satisfies that $y(\widehat{\Delta}(k\Lam_0))\subset \widehat{\Delta}_+^{re}$.
We claim that 
we may assume that
$\bar y(\theta)\in \Delta_+$.
Indeed,
$Pr_{k,y}=Pr_{k,y t_{q\varpi_j}\bar \pi_j}$
by \eqref{eq:Pr=Pr} for $j\in J$, and
if $\bar y(\theta)\in \Delta_-$,
then there exists $j\in J$ such that
$\bar y \bar \pi_j(\theta)\in \Delta_+$
by Lemma \ref{Lem:theta}.
We obtain  $\eta\in P_+^{\vee}$ and $\theta(\eta)\leq q-1$,
which completes the proof
 by Proposition \ref{Pro:equiv-class-adm}.
\end{proof}

\begin{Pro}\label{Pro:dimension-of-orbit}
Let $\lam\in \overline{Pr}_{k,t_{-\eta}}$
with $\eta\in P_+^{\vee}$, $\theta(\eta)\leq q-1$.
Then
\begin{align*}
\dim \on{Var}(J_{\lam})=|\Delta|-|\Delta( \lam)|,
\end{align*}
where $\Delta(\lam)=\{\alpha\in \Delta\mid \lam(\alpha^{\vee})\in \Z\}$.
\end{Pro}
\begin{proof}
We have 
${\Delta}(\lam)=t_{-\eta}(\widehat{\Delta}(k\Lam_0))\cap\Delta
=\{\alpha\in \Delta\mid 
\alpha(\eta)=0\}$.
Since $\eta\in P_+^{\vee}$,
it follows that ${\Delta}(\lam)$
 a subroot system of $\Delta$
generated by the simple roots $\alpha_i$ such that
$\alpha_i(\eta)=0$.
Let $\mf{p}$ be the corresponding parabolic subalgebra 
of $\fing$ containing $\mf{b}_+=\h\+\mf{n}$,
$\mf{m}$ its nilradical, 
$\mf{l}$ its Levi subalgebra.
Denote by $L_{\mf{l}}(\lam)$ the simple finite-dimensional module of $\mf{l}$ with highest weight $ \lam$.
Then \cite{Jan77}
\begin{align}
L(\lam)=U(\fing)\*_{U(\mf{p})}L_{\mf{l}}(\lam),
\end{align}
where $L_{\mf{l}}( \lam)$ is regarded as a $\mf{p}$-module by the projection $\mf{p}\ra \mf{l}$.
It follows that the Gelfand-Kirillov dimension
$\on{Dim}L( \lam)$ of $L( \lam)$ equals to $\dim \mf{m}=1/2(|\Delta|-|\Delta( \lam)|)$.
As $L( \lam)$ is holonomic,
we have
$\dim \on{Var}(J_{\lam})=2 \on{Dim}L(\bar \lam)=|\Delta|-|\Delta( \lam)|$ as required. \end{proof}

Let $\mathbb{O}_{min}$ be the unique minimal non-trivial nilpotent orbit of $\fing$,
which is of dimension $2h^{\vee}-2=2(n-1)$ \cite{Wan99}.
We have $\mathbb{O}_{min}=G.e_{\alpha}$,
where $e_{\alpha}$ is a root vector of some root $\alpha$.
\begin{Pro}\label{Pro:minimal}
 Let $\fing=\mf{sl}_n$,
 $k$ an admissible number with denominator $q$.
 Then
 \begin{align*}
  [{\overline{Pr}_k^{\mathbb{O}_{min}}}]
  =\bigsqcup_{a=1}^{q-1}[
  \overline{Pr}_{k,t_{-a\varpi_1}}]
  =\bigsqcup_{a=1}^{q-1}\{[\bar
  \lam-\frac{ap}{q}\varpi_1]\mid \lam\in \wh{P}_+^{p-n}\}.
 \end{align*}
 \end{Pro}
\begin{proof}
Let $\lam\in \overline{Pr}^k_{t_{-\eta}}$
with $\eta\in P_+^{\vee}$, $\theta(\eta)\leq q-1$.
By Proposition \ref{Pro:dimension-of-orbit},
 $\lam\in \overline{Pr}_k^{\mathbb{O}_{min}}$
 if and only if
 $|\Delta(\lam)|=n(n-1)-2(n-1)=(n-1)(n-2)$.
This happens if and only if
$\Delta( \lam)$ is a root system of type $A_{n-2}$,
or equivalently,
$\eta=-a\varpi_1$ or $\eta=-a\varpi_{n-1}$ for some $a=1,\dots,q-1$.
Since $ [\overline{Pr}_{k,t_{-a\varpi_1}}]= [\overline{Pr}_{k,t_{-a\varpi_{n-1}}}]$ by \eqref{eq:Pr=Pr}
and Proposition \ref{Pro:equiv-class-adm},
we get that
\begin{align}
  [{\overline{Pr}_k^{\mathbb{O}_{min}}}]
  =\bigcup_{a=1}^{q-1}[
 \overline{ Pr}_{k,t_{-a\varpi_1}}].
 \end{align}
It is straightforward to see that this is a disjoint union.
\end{proof}
Note that if $M$ is a cuspidal 
admissible $\g$-module of level $k$ with finite dimensional weight spaces
then 
$\on{Ann}_{U(\fing)}(M)=J_{ \lam}$ for some $\lam\in  [{\overline{Pr}_k^{\mathbb{O}_{min}}}]$
(\cite{Mat00}).
\subsection{The $\mf{sl}_3$-cases}
 Let $\fing=\mf{sl}_3$.
 Then
the level $k$ is admissible if and only if
 $k+3=p/q$ with $p,q\in \N$, $(p,q)=1$, $p\geq 3$.
 We have
 \begin{align}
  \mathbb{O}_q=\begin{cases}
		\mathbb{O}_{prin}&\text{if }q\geq 3,\\
		\mathbb{O}_{min}&\text{if }q=2,\\
		{0}&\text{if }q=1.
	       \end{cases}
 \end{align}
 Note that $\dim \mathbb{O}_{prin}=6$,
 $\dim \mathbb{O}_{min}=4$.

The following assertion follows immediately from 
Proposition \ref{Pro:principal} and Proposition \ref{Pro:minimal}.
 
\begin{Pro}\label{Pro: Cases for sl_3}
Let $k$ be an admissible number for $\affg$,
so that
$k+3= p/q$, $p,q\in \Z$, $p\geq 3$, $(p,q)=1$.
We have
\begin{align*}
&[\overline{Pr}_k^{\mathbb{O}_{prin}}]=
\{[\bar \lam-\frac{p}{q}(\bar \mu+\rho)]\mid \lam\in \widehat{P}_+^{p-3},
\mu\in \widehat{P}_+^{q-3} \},\\
&[\overline{Pr}_k^{\mathbb{O}_{min}}]=\{[\bar \lam-\frac{ap}{q}\varpi_1]\mid \lam\in \widehat{P}_+^{p-3},\
a=1,\dots, q-1
\},\\
&[\overline{Pr}_k^{\{0\}}]=\{[\bar \lam]\mid \lam\in \widehat{P}_+^{p-3}\}.
\end{align*}
\end{Pro}

\subsection{Compatibility of admissible $\fing$-modules with the restriction}
 Let $\mathfrak{p}=\mathfrak{l}+\mathfrak{m}$ be a parabolic subalgebra of $\mathfrak{g}$,
 where $\mf{l}$ is its Levi subalgebra and $\mf{m}$ is its nilradical.
Let $[\mathfrak{l},\mathfrak{l}]=\sum_{i} \mathfrak{l}_i$ be a decomposition into the sum of 
simple Lie subalgebras.
Let $\widehat{\mf{l}}_i$ be the affine Kac-Moody algebra  associated with $\mf{l}_i$.
For an admissible number $k$ of $\affg$,
define the admissible number $k_i$ for $\widehat{\mf{l}}_i$ 
by
\begin{align}
k_i+h_i^{\vee}=\frac{2}{(\theta_i,\theta_i)}(k+h^{\vee}),
\end{align}
where $h_i^{\vee}$ is the dual Coxeter number of $\mf{l}_i$,
$\theta_i$ is the highest root of $\mf{l}_i$.

\begin{Th}\label{Th:restriction}
Let  $M$ be an admissible $\fing$-module of level $k$.
Then $\mathfrak{m}$-invariant subspace $M^{\mathfrak{m}}$ is an admissible  module of $\mathfrak{l}_i$ of level $k_i$ for each $i$.
\end{Th}
\begin{proof}
Since $M$ is an $A(\Vs{k})$-module,
$\mathbf{L}(M)$ is a $\Vs{k}$-module.
Hence by \cite[Theorem 7.5]{A-BGG},
the semi-infinite cohomology space
$H^{\frac{\infty}{2}+n}(\mf{m}[t,t^{-1}],\mathbf{L}(M))$
is a module over the admissible affine vertex algebra
$V_{k_i}(\mf{l}_i)$ for all $n$, $i$.
Therefore
the top weigh space of $H^{\frac{\infty}{2}+n}(\mf{m}[t,t^{-1}],\mathbf{L}(M))$
is a module over 
$A(V_{k_i}(\mf{l}_i))$-module,
that is to say,
an admissible $\mf{l}_i$-module of level $k_i$.
The assertion follows since 
 the top weight space of 
$H^{\frac{\infty}{2}+n}(\mf{m}[t,t^{-1}],\mathbf{L}(M))$, $n\geq0$,
is isomorphic to $H^{n}(\mf{m}, M)$ as an $\mf{l}$-module
provided that $H^{n}(\mf{m}, M)$ is nonzero
(see the proof of \cite[Lemma 4.3]{A12-2}).

\end{proof}

\section{Gelfand-Tsetlin modules}

We recall the definition of a Gelfand-Tsetlin module for $\gl_n$.
Let  $U=U(\gl_n)$ and $\{E_{ij}\; | \; 1\leq i,j \leq n\}$  be the
standard basis of $\mathfrak{gl}_{n}$ of elementary matrices.  For each $m\leqslant n$ let $\mathfrak{gl}_{m}$ be the Lie subalgebra
of $\gl_n$ spanned by $\{ E_{ij}\,|\, i,j=1,\ldots,m \}$. Then we have the following chain
$$\gl_1\subset \gl_2\subset \ldots \subset \gl_n,$$
which induces  the chain $U_1\subset$ $U_2\subset$ $\ldots$ $\subset
U_n$ of the universal enveloping algebras  $U_{m}=U(\gl_{m})$, $1\leq m\leq n$. Let
$Z_{m}$ be the center of $U_{m}$. Then $Z_m$ is the polynomial
algebra in the $m$ variables $\{ c_{mk}\,|\,k=1,\ldots,m \}$,
\begin{equation}\label{equ_3}
c_{mk } \ = \ \displaystyle {\sum_{(i_1,\ldots,i_k)\in \{
1,\ldots,m \}^k}} E_{i_1 i_2}E_{i_2 i_3}\ldots E_{i_k i_1},
\end{equation}
$m=1, \ldots, n$.
Following \cite{DFO94}, we call the subalgebra of $U$ generated by $\{
Z_m\,|\,m=1,\ldots, n \}$ the \emph{(standard) Gelfand-Tsetlin
subalgebra} of $U$ and  denote it by  ${\Ga}$. In fact,  ${\Ga}$ is the polynomial algebra in the $\displaystyle \frac{n(n+1)}{2}$ variables $\{
c_{ij}\,|\, 1\leqslant j\leqslant i\leqslant n \}$ (\cite{Zh74}). 
\begin{Def}
\label{definition-of-GZ-modules} A finitely generated $U$-module
$M$ is called a \emph{Gelfand-Tsetlin module (with respect to
$\Gamma$)} if $M$ splits into  a direct sum
of $\Ga$-modules:

\begin{equation}
M=\bigoplus_{\sm\in\Sp\Gamma}M(\sm),
\end{equation}
where $$M(\sm)=\{v\in M| \sm^{k}v=0 \text{ for some }k\geq 0\}.$$
\end{Def}

All those $\sm\in\Sp\Gamma$ for which $M(\sm)\neq 0$ form the \emph{Gelfand-Tsetlin support} of $M$. The dimension of the subspace  $M(\sm)$ is the \emph{Gelfand-Tsetlin multiplicity} of $\sm$. 
Since the Gelfand-Tsetlin subalgebra contains the Cartan subalgebra $\h$ spanned by  $\{ E_{ii}\,|\, i=1,\ldots, n\}$ then any simple Gelfand-Tsetlin module is a weight module. 
Note that every weight module with finite dimensional weight subspaces is a Gelfand-Tsetlin module. Indeed, any weight space is invariant under the action of the Gelfand-Tsetlin subalgebra (since $\Gamma$ contains all $E_{ii}$'s) and hence decomposes 
into its generalized eigen subspaces. 
In particular, every highest weight module is a Gelfand-Tsetlin module.

\subsection{Generic Gelfand-Tsetlin modules}

\begin{Def} For a vector $v$ in $\mathbb{C}^{\frac{n(n+1)}{2}}$, by $T(v)$ we will denote the following array with entries $\{v_{ij}:1\leq j\leq i\leq n\}$ 
\begin{center}

\Stone{\mbox{ \scriptsize {$v_{n1}$}}}\Stone{\mbox{ \scriptsize {$v_{n2}$}}}\hspace{1cm} $\cdots$ \hspace{1cm} \Stone{\mbox{ \scriptsize {$v_{n,n-1}$}}}\Stone{\mbox{ \scriptsize {$v_{nn}$}}}\\[0.2pt]
\Stone{\mbox{ \scriptsize {$v_{n-1,1}$}}}\hspace{1.5cm} $\cdots$ \hspace{1.5cm} \Stone{\mbox{ \tiny {$v_{n-1,n-1}$}}}\\[0.3cm]
\hspace{0.2cm}$\cdots$ \hspace{0.8cm} $\cdots$ \hspace{0.8cm} $\cdots$\\[0.3cm]
\Stone{\mbox{ \scriptsize {$v_{21}$}}}\Stone{\mbox{ \scriptsize {$v_{22}$}}}\\[0.2pt]
\Stone{\mbox{ \scriptsize {$v_{11}$}}}\\
\medskip
\end{center}
Such an array will be called a \emph{Gelfand-Tsetlin tableau} of height $n$. 
\end{Def}
 A classical result  of Gelfand and Tsetlin \cite{GT50} gives an 
explicit realization of any simple finite dimensional $\mathfrak{gl}_n$-module in terms of Gelfand-Tsetlin tableaux. Using these Gelfand-Tsetlin formulas one can defined the class of infinite dimensional generic modules for $\mathfrak{gl}_n$, see for example \cite{DFO94}.

\begin{Def} \label{def-gen}
A vector $v \in {\mathbb C}^{\frac{n(n+1)}{2}}$ and the Gelfand-Tsetlin tableau $T(v)=T(v_{ij})$ are called \emph{generic} if
$v_{ki}-v_{kj}\notin\mathbb{Z}$ for all $1\leq i\neq j \leq k \leq n-1$. It is {\it strongly generic} if $v_{rs}-v_{rt}\notin \mathbb{Z}$ for any  $1\leq s\neq t \leq r \leq n$.
\end{Def}

\begin{Th}[{\cite[Section 2.3]{DFO94}}]\label{Generic Gelfand-Tsetlin modules}
Let $T(v) = T(v_{ij})$ be a generic  Gelfand-Tsetlin tableau of height $n$. Denote by ${\mathcal B}(T(v))$  the set of all Gelfand-Tsetlin tableaux $T(v') = T(v'_{ij})$ satisfying $v'_{nj}=v_{nj}$, $v'_{ij}-v_{ij}\in\mathbb{Z}$ for , $1\leq j\leq i \leq n-1$. 
\begin{itemize}
\item[(i)] The vector space $V(T(v)) = \Span {\mathcal B}(T(v))$ has  a structure of a $\mathfrak{gl}_n$-module with action of the generators of $\mathfrak{gl}_n$ given by the Gelfand-Tsetlin formulas.

$$E_{k,k+1}(T(v))=-\sum_{i=1}^{k}\left(\frac{\prod_{j=1}^{k+1}(v_{ki}-v_{k+1,j})}{\prod_{j\neq i}^{k}(v_{ki}-v_{kj})}\right)T(v+\delta^{ki}),$$

$$E_{k+1,k}(T(v))=\sum_{i=1}^{k}\left(\frac{\prod_{j=1}^{k-1}(v_{ki}-v_{k-1,j})}{\prod_{j\neq i}^{k}(v_{ki}-v_{kj})}\right)T(v-\delta^{ki}),$$

$$E_{kk}(T(v))=\left(k-1+\sum_{i=1}^{k}v_{ki}-\sum_{i=1}^{k-1}v_{k-1,i}\right)T(v).$$

\noindent 
where $\delta^{ij} \in \mathbb{C}^{\frac{n(n-1)}{2}}$ is defined by $(\delta^{ij})_{ij}=1$ and all other $(\delta^{ij})_{k\ell}$ are zero. 
\item[(ii)] The $\mathfrak{gl}_n$-module $V(T(v))$ is a Gelfand-Tsetlin module with all non-zero Gelfand-Tsetlin multiplicities equal to $1$. 
\end{itemize}
\end{Th}


\begin{Def} For any Gelfand-Tsetlin tableau $T(v)$ of height $n$, we associate the following sets:
\begin{align} 
\Omega(T(v))&:=\{(r,s,t):v_{rs}-v_{r-1,t}\in\mathbb{Z}\}\\
\Omega^{+}(T(v))&:=\{(r,s,t) :v_{r,s}-v_{r-1,t}\in\mathbb{Z}_{\geq 0}\}.
\end{align}
\end{Def}

\begin{Th}[{\cite[Theorem 6.14]{FGR15}}]\label{characterization of irreducible basis} Let $T(v) = T(v_{ij})$ be a generic  Gelfand-Tsetlin tableau of height $n$.
The simple subquotient of $V(T(v))$ which contains a tableau $T(v')\in \mathcal{B}(T(v))$ has the following basis of tableaux:
$$\mathcal{I}(T(v'))=\{T(v'')\in\mathcal{B}(T(v)): \Omega^{+}(T(v''))=\Omega^{+}(T(v'))\}.$$
The action of $\mathfrak{gl}_n$ on this simple module is given by the classical Gelfand-Tsetlin formulas.
\end{Th}

\begin{Co}\label{Co: Basis for generic Verma} If $\lambda=(\lambda_{1},\ldots,\lambda_{n})$ is a $\mathfrak{gl}_n$ weight such that $\lambda_{i}-\lambda_{j}\notin \mathbb{Z}$ then, the Verma module $M(\lambda)$ is simple and admits a tableau realization as a subquotient of a generic module $V(T(v))$ with basis 
$$\mathcal{I}(T(v))=\{T(v')\in\mathcal{B}(T(v)):\Omega^{+}(T(v'))=\{(r,s,s):1\leq s\leq r\leq n\}\}$$

where $c_{i}=\lambda_{i}-i+1$ and $T(v)$ is the generic tableau with entries $v_{rs}=c_{s}$.
\end{Co}

\section{Admissible generic modules in the principal nilpotent orbit}
We set $G:=S_n\times S_{n-1} \times \cdots S_1$ and the $i$-th component of $\sigma \in G$ will be denoted by $\sigma[i]$. Note that for any $\sigma=(\sigma[n],\sigma[n-1],\ldots,\sigma[1]) \in G$ we have $V(T(v))\simeq V(T(\sigma(v)))$. We immediately have
\begin{Lem}\label{strongly generic and Verma conditions}
Let $T(v)$ be any strongly generic Gelfand-Tsetlin tableau. There exists $\sigma \in G$ such that $\Omega(T(\sigma(v)))\subseteq\{(r,s,s):1\leq s\leq r\leq n\}$. 
\end{Lem}

In \cite{Kho05} is constructed an embedding of $U(\mathfrak{gl}_n)$ into a certain localization $\mathcal{D}_{m}$ of the Weyl algebra $\mathcal{A}_{m}$, for $m=n(n+1)/2$ with generators $\{x_{ij},\partial_{ij}:1\leq j\leq i\leq n\}$ (for details see \cite[Theorem 3.2]{Kho05}) based on Gelfand-Tsetlin formulas. Denote by $\phi$ this embedding and $\mathcal{P}$ the polynomial algebra in variables $y_{ij}=\partial_{ij}x_{ij}$. This embedding essentially provides a tableaux realization.
As an application of the construction of this embedding  an alternative proof of Duflo's Theorem (which describes the annihilators of Verma modules c.f,\ \S \ref{subsection:Admissible representations and nilpotent orbits}) for $\mathfrak{gl}_n$ is given. In particular,  the description of annihilators of generic Verma modules is given in (\cite[Theorem 4.1]{Kho05}) together with their explicit tableaux realization  (Corollary \ref{Co: Basis for generic Verma}). In order to describe annihilators of simple subquotients of strongly generic modules we will use a slight modification of the proof of Theorem 4.1 in \cite{Kho05}.

\begin{Th}\label{Strongly generic submodule has trivial annihilator}
Let $T(v)$ be any strongly generic Gelfand-Tsetlin tableau. For any simple subquotient $N$ of $V(T(v))$ we have $$Ann_{U(\mathfrak{gl}_n)}N=U(\mathfrak{gl}_n)Ann_{Z(\mathfrak{gl}_n)}N.$$ 
\end{Th}

\begin{proof}
By Lemma \ref{strongly generic and Verma conditions} we can assume without loss of generality that $\Omega(T(v))\subseteq\{(r,s,s):1\leq s\leq r\leq n\}$. Let $N$ be a simple subquotient of $V(T(v))$, by Theorem \ref{characterization of irreducible basis}, there exist $A\subseteq\Omega(T(v))$ such that $N$ is spanned by the set of tableaux $\mathcal{I}(N)=\{ T(w)\in\mathcal{B}(T(v)) :\ \Omega^{+}(T(w))=A\}$. 

Let $I=U(\mathfrak{gl}_n)Ann_{Z(\mathfrak{gl}_n)}N$. It is clear that $I\subseteq Ann_{U(\mathfrak{gl}_n)}N$. To prove the opposite inclusion, take some $u\in U(\mathfrak{gl}_n)$ such that $uN=0$. The image of $u$ under the inclusion $\phi$ can be written as $\phi(u)=\sum_{k=1}^{K}D_{k}f_{k}$ for some $K\in\mathbb{Z}_{\geq 0}$, $f_{k}$ in certain localization $\tilde{\mathcal{P}}$ of $\mathcal{P}$ and $D_{k}=\prod_{i=1}^{n-1}\prod_{j=1}^{i}\partial_{ij}^{n_{ijk}}$ with  $n_{ijk}\in\mathbb{Z}$. 
For $w\in\{v+z:z\in\mathbb{Z}^{\frac{n(n-1)}{2}}\}$, let $I_{w}\subset \mathcal{P}$ be the ideal of $\mathcal{P}$ generated by $y_{ij}-w_{ij}$, for $1\leq j\leq i\leq n$. One gets that $I_{w}\tilde{\mathcal{P}}$ is a proper maximal ideal of $\tilde{\mathcal{P}}$ for any $w$ such that $T(w)\in B$. For each rational function $f\in \tilde{\mathcal{P}}$ and $T(w)\in B$ define the evaluation $f(w)$ as the image of $f$ under the canonical projection $\tilde{\mathcal{P}}\to \tilde{\mathcal{P}}/I_{w}\tilde{\mathcal{P}}\simeq\mathbb{C}$. 
Let 
$$L=\max\{|n_{ijk}| : 1\leq k\leq K, 1\leq i\leq n,1\leq j\leq i\},$$
$$B=\{T(w)\in \mathcal{I}(N):|w_{rs}-w_{r-1,s}|>L,\text{ for any } (r,s,s)\in \Omega(T(v))\}$$

By construction, one easily obtains that for all $T(w)\in B$ the action of $u$ on $T(w)$ is given by 
\begin{align}
uT(w)=\sum_{k=1}^{K}f_{k}(w)T\left(w+\sum_{i=1}^{n-1}\sum_{j=1}^{i}n_{ijk}\delta^{ij}\right).
\end{align}
Due to the choice of $T(w)\in B$ we have $T\left(w+\sum_{i=1}^{n-1}\sum_{j=1}^{i}n_{ijk}\delta^{ij}\right)\in \mathcal{I}(N)$ for any $k$, which implies  $uT(w)=0$ if and only if $f_{k}\in I_{w}\tilde{\mathcal{P}}$ for all $1\leq k\leq K$. Finally, by using just the fact that the module is generic, it is shown in (\cite[Theorem 4.1]{Kho05}) that $f_{k}\in I_{w}\tilde{\mathcal{P}}$ for all $1\leq k\leq K$ implies $u\in I$.
\end{proof}

\subsection{Admissible generic Gelfand-Tsetlin modules
in the principal nilpotent orbit for $\mathfrak{sl}_n$}
Let 
$k$ be an admissible number for $\widehat{\mf{sl}}_n$,
so that 
\begin{align}
k+n=p/q,\quad  p\geq n,\ q\geq 1,\ (p,q)=1.
\end{align}
By Proposition \ref{Pro:principal},
an element of 
$[\overline{Pr}_k^{\mathbb{O}_{prin}}]$
for $\fing=\mathfrak{sl}_{n}$
is represented by the  weight of the form 
\begin{align}
\lambda'=\lambda-\frac{p}{q}(\mu+\rho)=\left(\lambda_{1}-\frac{p}{q}(\mu_{1}+1),\lambda_{2}-\frac{p}{q}(\mu_{2}+1),\ldots,\lambda_{n-1}-\frac{p}{q}(\mu_{n-1}+1)\right),
\end{align}
where $\lambda_{i},\mu_{i}\in\mathbb{Z}_{\geq 0}$ are such that $\lambda_{1}+\ldots+\lambda_{n-1}\leq p-n$, $\mu_{1}+\ldots+\mu_{n-1}\leq q-n$. 
(The set $[\overline{Pr}_k^{\mathbb{O}_{prin}}]$ is empty if $q<n$.)

A $\mathfrak{gl}_n$-weight corresponding to $\lambda'$ is $(b_{1}+b,\ldots,b_{n}+b)$, where for $1\leq k\leq n-1$, we take $b_{i}:=\sum_{k=i}^{n-1}(\lambda_{k}-\frac{p}{q}(\mu_{k}+1))$, $b_{n}=0$ and $b=-\frac{1}{n}\sum_{i=1}^{n}b_{i}$. Note that for any $i<j$, $b_{i}-b_{j}=\sum_{k=i}^{j-1}(\lambda_{k}-\frac{p}{q}(\mu_{k}+1))\notin\mathbb{Z}$. In particular, the corresponding simple highest weight module associated with $\lambda'$ admits a tableaux realization as a subquotient of $V(T(v))$ with the strongly generic  $T(v)$, where $v_{ij}=b_{j}+b-j+1$ for any $1\leq j\leq i\leq n$.\\

The highest weight modules with the same annihilator as $L(\lambda')$ are all of the form $L(w\circ \lambda')$ with $w$ an element of the Weyl group. All these highest weight modules are simple Verma modules associated with some $\mathfrak{gl}_n$-weight of the form $(a_1,a_{2},\ldots,a_{n})$ satisfying $a_{i}-a_{j}\notin\mathbb{Z}$ for any $i\neq j$. Thus, in order to describe all simple Gelfand-Tsetlin modules with the same annihilator as $L(\lambda')$ it is enough to describe simple Gelfand-Tsetlin modules with the same annihilator as a generic Verma module obtained by a tableaux realization as a subquotient of the module $V(T(v))$ with a strongly generic tableau 
$T(v)$ with entries $
v_{ij}:=c_j$  for any $1\leq j\leq i\leq n$.
 
 Hence, Theorem \ref{Strongly generic submodule has trivial annihilator} immediately implies 
 
\begin{Co}\label{cor-generic}
Let $\lambda'$ be an admissible $\mathfrak{sl}_n$-weight of level $k$ and $(a_{1},\ldots,a_{n})$ any corresponding $\mathfrak{gl}_n$-weight. 
Then any   simple subquotient of a strongly generic module $\{V(T(v)):v_{ni}=a_{i}-i+1\}$ is a simple admissible Gelfand-Tsetlin module in the principal orbit. 
Moreover, any simple admissible generic Gelfand-Tsetlin module of level $k$ in the principal nilpotent orbit with central character determined by the $\mathfrak{gl}_n$-highest weight $(a_{1},\ldots,a_{n})$ is isomorphic to a subquotient of 
some strongly generic module $\{V(T(v)):v_{ni}=a_{i}-i+1\}$. 
\end{Co}

Now applying Theorem \ref{characterization of irreducible basis} we obtain a complete classification of all simple admissible generic Gelfand-Tsetlin modules of level $k$ in the principal orbit together with their explicit construction
 via the  Gelfand-Tsetlin formulas.

\begin{Rem}
\begin{itemize}
\item
We observe that there exist  simple generic Gelfand-Tsetlin modules which  are not in  the principal orbit.
\item Among the modules described in Corollary \ref{cor-generic} there are those with finite weight multiplicities. Such modules are either highest weight modules or induced from an $\mathfrak{sl}_2$-subalgebra. The latter modules will be also discussed in  Theorem \ref{Th: Admissible induced by sl(2)}. 

\end{itemize}
\end{Rem}

\section{Classification of simple admissible Gelfand-Tsetlin $\mathfrak{sl}_3$-modules}
Let $k$ be an admissible number for $\g=\mf{sl}_3$,
so that 
\begin{align}
 k+3=p/q,\quad p\geq 3,\ q\geq 1, \ (p,q)=1.
\end{align}
In this section we  classify all simple Gelfand-Tsetlin
admissible $\mf{sl}_3$-module of an admissible level $k$.
Using the classification of simple Gelfand-Tsetlin $\mathfrak{sl}_3$-modules \cite{FGR}, where explicit basis and the action of the generators of $\mathfrak{sl}_3$ are given, we are able to present explicitly all simple Gelfand-Tsetlin modules with necessary annihilator.

.
\subsection{Singular Gelfand-Tsetlin $\mathfrak{sl}_3$-modules}

Let $v_{rs}$, $1\leq s\leq r\leq 3$ complex numbers, $v=(v_{31},v_{32},v_{33},v_{21},v_{22},v_{11})$ and $w=(m,n,k):=(m_{21},m_{22},m_{11})\in\mathbb{Z}^{3}$. We will denote by $T(v)$ and $T(v+w)$ the following tableaux:

\begin{center}
 \hspace{1.6cm}\Stone{$v_{31}$}\Stone{$v_{32}$}\Stone{$v_{33}$}\hspace{1cm}\Stone{$v_{31}$}\Stone{$v_{32}$}\Stone{$v_{33}$}\\[0.2pt]
 $T(v)$:= \hspace{0.3cm}\Stone{$v_{21}$}\Stone{$v_{22}$} \hspace{0.4cm} $T(v+w)$:= \Stone{$v_{21}+m$}\Stone{$v_{22}+n$}\\[0.2pt]
 \hspace{1.3cm}\Stone{$v_{11}$}\hspace{3.8cm}\Stone{$v_{11}+k$}\\
\end{center}

Also, set $\mathcal{B}(T(v)):=\{T(v+w): w\in\mathbb{Z}^{3}\}$

\begin{Def} If $T(v')$ is any Gelfand-Tsetlin tableau from the set $\mathcal{B}(T(v))$. Then
\begin{itemize}
\item[(i)] $T(v')$ is  \emph{singular}, if $v'_{21}-v'_{22}\in\mathbb{Z}$.
\item[(ii)] $T(v')$ is  \emph{critical}, if $v'_{21}-v'_{22}=0$.
\end{itemize}
\end{Def}

We will follow \cite{FGR} to construct an explicit basis of tableaux for simple singular Gelfand-Tsetlin $\sl_3$-modules starting with a critical tableau.\\

From now on we will denote by $T(\bar{v})$ a fixed critical tableau and by $T(v)$ a generic tableau. Denote by $\tau$ the transposition $(1,2)$. We formally introduce  a new tableau  ${\mathcal D} T({\bar{v}} + w)$ for  every $w \in {\mathbb Z}^{3}$ subject to the relations ${\mathcal D} T({\bar{v}} + w) + {\mathcal D} T({\bar{v}} + \tau(w)) = 0$. We call ${\mathcal D} T(\bar{v}+w)$  {\it the derivative Gelfand-Tsetlin tableau} associated with $w$. 

\begin{Def}
Set  $V(T(\bar{v}))$ to be the vector space spanned by the set of tableaux $\{ T(\bar{v} + w), \,\mathcal{D} T(\bar{v} + w) \; | \; w \in {\mathbb Z}^{3}\}$, subject to the relations $T(\bar{v} + w) = T(\bar{v} +\tau(w))$ and ${\mathcal D} T(\bar{v} + w) + {\mathcal D} T(\bar{v} +\tau(w)) = 0$.  Choose a basis of $V(T(\bar{v}))$ to be the  set $\{T(\bar{v}+w),\ \mathcal{D} T(\bar{v}+w')\ :\ w,w'\in\mathbb{Z}^{3} \text{ and } w_{21}-w_{22}\in\mathbb{Z}_{\geq 0},\ w'_{21}-w'_{22}\in\mathbb{Z}_{< 0}\}$.
\end{Def}

For $v=(v_{31},v_{32},v_{33},v_{21},v_{22},v_{11})$ and a rational function $f$  on variables $v_{rs}$  which is smooth on the hyperplane $v_{21}-v_{22}=0$,  define  a linear map 
\begin{align}
\mathcal{D}^{\bar{v}} (f T(v+z)) = \mathcal{D}^{\bar{v}} (f) T(\bar{v}+z) +   f(\bar{v}) \mathcal{D} T(\bar{v}+z),
\end{align}
where $\mathcal{D}^{\bar{v}}(f) = \frac{1}{2}\left(\frac{\partial f}{\partial v_{21}}-\frac{\partial f}{\partial v_{22}}\right)(\bar{v})$.

\begin{Th}[\cite{FGR}]For any critical tableau $T(\bar{v})$, the vector space $V(T(\bar{v}))$ together with the action of $\gl_3$ given by 
\begin{align*}
E_{rs}(T(\bar{v} + z))=&\  \mathcal{D}^{\bar{v}}((v_{21} - v_{22})E_{rs}(T(v + z)))\\
E_{rs}(\mathcal{D}T(\bar{v} + z')))=&\ \mathcal{D}^{\bar{v}} ( E_{rs}(T(v + z'))),
\end{align*}
for all $z,z'\in\mathbb{Z}^{3}$ with $z'\neq\tau(z')$,
has a structure of a Gelfand-Tsetlin module over $\mathfrak{gl}_3$. Here, the action $E_{rs}(T(v + z))$ is given by the classical Gelfand-Tsetlin formulas as in the generic case.
\end{Th}

\begin{Rem}\label{Rem: 1-singular modules}
This construction can be extended to $\mathfrak{gl}_n$ by considering Gelfand-Tsetlin tableaux $T(v)$ with exactly one singularity (i.e. when exists a unique $(k,i,j)$ such that $v_{ki}-v_{kj}\in\mathbb{Z}$). Such modules were studied in \cite{FGR16}. 
\end{Rem}

\subsection{Construction of simple admissible Gelfand-Tsetlin modules for $\mathfrak{sl}_3$}
In this subsection we will present explicitly simple subquotiens of $V(T(v))$ (respectively, of $V(T(\bar{v}))$) for generic $v$ (respectively, for critical $\bar{v}$) that have the same annihilator as  admissible highest weight modules $L(\lambda)$. 

In order to simplify the description we will use the following notation:
\begin{align}
Tab(w):=\begin{cases}
T(v+w), & \mbox{if } \ v\text{\ is generic},\\
T(\bar{v}+w), & \mbox{if } \bar{v}\text{\ is critical and}\ w_{ki}-w_{kj}\leq 0,\\
\mathcal{D} T(\bar{v}+w), & \mbox{if } \bar{v}\text{\ is critical and}\ w_{ki}-w_{kj}> 0.
\end{cases}
\end{align}


If $D_1,...,D_k$ are some sets of inequalities we denote
\begin{align}
L(D_1\cup...\cup D_k) = \Span \{ Tab(z)\,|\,z \mbox{ satisfies one of } D_1,...,D_k\}.
\end{align}

\subsubsection{Twisted localization}
 
Let  $U = U(\mathfrak{sl}_3)$. 
We recall the definition of the localization functor on $U$-modules.
For details we refer the reader to \cite{Deo80} and \cite{Mat00}.\\
For every $\alpha \in \Delta$ we fix an $\mathfrak{sl} (2)$-triple $(e_{\alpha}, f_{\alpha}, h_{\alpha})$, i.e. $e_{\alpha} \in {\mathfrak g}_{\alpha}$, $f_{\alpha} \in {\mathfrak g}_{-\alpha}$ and $[e_{\alpha}, f_{\alpha}]= h_{\alpha}$.  Since $\mbox{ad}\, f_\alpha$ acts locally finitely on $U$, the multiplicative set
${\bf F}_{\alpha}:=\{ f_\alpha^n \; | \; n \in {\mathbb Z}_{\geq 0} \} \subset U$
satisfies Ore's localizability conditions. Denote by
$\mathcal{D}_\alpha U$  the localization of $U$ relative to ${\bf F}_{\alpha}$.
 If $M$ is a weight module the denote by $\mathcal{D}_\alpha M$ the {\it
$\alpha$--localization of $M$}, defined as $\mathcal{D}_\alpha M =
\mathcal{D}_\alpha U \otimes_U M$. If $f_\alpha$ is injective on $M$, then $M$ can be naturally viewed as
a submodule of $\mathcal{D}_\alpha M$. 

 For $a \in \C$ and $u \in \mathcal{D}_\alpha U$   set
\begin{equation} \label{theta}
\Theta_a(u):= \sum_{i \geq 0} \binom{a}{i}\,
( \mbox{ad}\, f_\alpha)^i (u) \, f_\alpha^{-i},
\end{equation}
where $\binom{a}{i}= \frac{a(a-1)...(a-i+1)}{i!}$. 
For a $\mathcal{D}_\alpha U$-module $M$ denote by $\Phi_\alpha^a M$ 
  the $\mathcal{D}_\alpha U$-module $M$ twisted by
the action
\begin{align}
u \cdot v^a := ( \Theta_a (u)\cdot v)^a,
\end{align}
where $u \in \mathcal{D}_\alpha U$, $v \in M$, and $v^a$ stands for the
element $v$ considered as an element of $\Phi_\alpha^a M$.
We set $f_{\alpha}^a \cdot v :=v^{-a}$ in $\Phi_\alpha^{-a} M$, $a
\in \C$.

The following lemma is standard.

\begin{Lem}
Let $M$ be a $\mathcal{D}_{\alpha} U$-module, $v \in M$, $u \in \mathcal{D}_\alpha U$ and $a,b \in \C$. Then
\begin{itemize}
\item[(i)] $\Phi_\alpha^a  M \simeq M$ whenever $a \in {\mathbb Z}$.

\item[(ii)] $\Phi_\alpha^a (
\Phi_\alpha^b M) \simeq \Phi_\alpha^{a+b}M $ and, consequently, $\Phi_\alpha^a
\circ\Phi_\alpha^{-a} M \simeq \Phi_\alpha^{-a} \circ\Phi_\alpha^{a} M \simeq M$.

\item[(iii)] $f_{\alpha}^a \cdot (f_{\alpha}^b \cdot v) =
f_{\alpha}^{a+b} \cdot v$.

\item[(iv)] $f_{\alpha}^a \cdot (u \cdot (f_{\alpha}^{-a}
\cdot v)) = \Theta_a(u) \cdot v$.
\end{itemize}
\end{Lem}

Let $M$ be any Gelfand-Tsetlin module and $\Phi_\alpha^a  M$ the twisted localization of $M$ for some $\alpha \in \Delta$ and $a \in \C$. It is clear that
 $Ann_{U}(M)=Ann_{U}(\Phi_\alpha^a  M)$.

\subsection{Minimal orbit.} 
By Proposition \ref{Pro: Cases for sl_3},
 elements of $[\overline{Pr}_k^{\mathbb{O}_{min}}]$
are represented by
$\mathfrak{sl}_{3}$-weights of the form 
\begin{align}
\lambda-\frac{ap}{q}\varpi_{1}=(\lambda_{1}-\frac{ap}{q},\lambda_{2}),
\end{align}
where 
$a,\ \lambda_{1},\lambda_{2}\in\mathbb{Z}_{\geq 0}$ are such that $\lambda_{1}+\lambda_{2}\leq p-3$ and $1\leq a\leq q-1$. 

First  we  identify simple highest weight modules whose annihilator coincides with the one of  $L(\lambda_{1}-\frac{ap}{q},\lambda_{2})$. They are the following modules:

\begin{itemize}
\item[(i)] $L(\lambda_{1}-\frac{ap}{q},\lambda_{2})$.
\item[(ii)] $L(\lambda_{2},\frac{ap}{q}-\lambda_{1}-\lambda_{2}-3)$.
\item[(iii)] $L(\frac{ap}{q}-\lambda_{1}-2,\lambda_{1}+\lambda_{2}-\frac{ap}{q}+1)$.
\end{itemize}

Note that these highest weight modules have bounded weight multiplicities of dimension at most $t=\lambda_{2}+1$. 

Next we describe simple Gelfand-Tsetlin modules with the same annihilator as  highest weight modules above.

 Fix $c:=\frac{\lambda_{2}+2\lambda_{1}-\frac{2ap}{q}}{3}$, $x:=\frac{\lambda_{2}-\lambda_{1}+\frac{ap}{q}}{3}-1$,  $t=\lambda_{2}+1$ (note that $c-x\notin\mathbb{Z}$) and consider any $y,z\in\mathbb{C}$ such that $\{c-z,x-z,c-y,x-y,z-y\}\cup\mathbb{Z}=\emptyset$.
 
\begin{Th}\label{Th: Modules associated with minimal orbit}
The following modules constitute a complete list of non-isomorphic simple admissible Gelfand-Tsetlin modules of level $k$
in the minimal orbit $\mathbb{O}_{min}$. All these modules have finite dimensional weight spaces.

\begin{itemize}
\item[(i)]  The following subquotients of the generic module $V(T(v))$ with $v=(c,x,x-t,c,x,c)$
{\scriptsize $$L_{1}:=L\left( \begin{split}
-t&<n\leq 0\\
&m\leq 0\\
&k\leq m
\end{split}\right);\ \ L_{2}=L\left( \begin{split}
&m\leq 0\\
-t&<n\leq 0\\
&m<k
\end{split};
\right);\ \ L_{3}=L\left( \begin{split}
&0<m\\
-t&<n\leq 0\\
&k\leq m
\end{split};\right);\ \ L_{4}=L\left( \begin{split}
&0<m\\
-t&<n\leq 0\\
&m<k
\end{split};
\right)$$}
These modules are highest weight modules (with respect to some Borel subalgebra) with weight multiplicities bounded by $t$.
\item[(ii)] Subquotients of the generic module $V(T(v_{1}))$ with $v_{1}=(c,x,x-t,c,x,z)$
{\scriptsize $$L_{5}=L\left( \begin{split}
-t &<n\leq 0\\
&m\leq 0
\end{split}\right);\ \ L_{6}=L\left( \begin{split}
-t &<n\leq 0\\
&0<m
\end{split}\right)$$}
 Both $L_{5}$ and $L_{6}$ are $\mathfrak{sl}_2$-induced modules  (with respect to opposite parabolic subalgebras  where $\mathfrak{sl}_2$ is generated by $E_{12}$, $E_{21}$).   Weight multiplicities of these modules are bounded by $t$. 
\item[(iii)] Generic cuspidal module with all weight multiplicities $t$ which is a subquotient of  $V(T(v_{2}))$ with $v_{2}=(c,x,x-t,y,x,z)$:
$$L_{7}=L\left( \begin{split}
-t<n\leq 0
\end{split}
\right)$$
\item[(iv)] $\mathfrak{sl}_2$-induced modules (with respect to opposite parabolic subalgebras  where $\mathfrak{sl}_2$ is generated by $E_{13}$, $E_{31}$) with weight multiplicities bounded by $t$ which are subquotients of the generic module $V(T(v_{3}))$ with $v_{3}=(c,x,x-t,z,x,z)$:
{\scriptsize $$L_{8}=L\left( \begin{split}
-t &<n\leq 0\\
&m<k
\end{split}
\right);\ \ L_{9}=L\left( \begin{split}
-t &<n\leq 0\\
&k\leq m
\end{split}\right)$$}

\item[(v)] The following subquotients of the generic module $V(T(u))$ with $u=(c,x,x-t,x,c,x)$:

{\scriptsize $$L_{10}=L\left( \begin{split}
-t &<m\leq 0\\
&n\leq 0\\
&k\leq m
\end{split}\right);\ L_{11}=L\left( \begin{split}
-t &<m\leq 0\\
&n\leq 0\\
&m<k
\end{split}
\right);\ L_{12}=L\left( \begin{split}
-t&<m\leq 0\\
&0<n\\
&k\leq m
\end{split}\right);\ L_{13}=L\left( \begin{split}
-t &<m\leq 0\\
&0<n\\
&m<k
\end{split}
\right)$$}
These are highest weight modules (with respect to some Borel subalgebra) with weight multiplicities bounded by $t$.
\item[(vi)] $\mathfrak{sl}_2$-induced modules (with respect to opposite parabolic subalgebras  where $\mathfrak{sl}_2$ is generated by $E_{23}$, $E_{32}$)  with weight multiplicities bounded by $t$ which are subquotients of the generic module $V(T(u_1))$ with $u_{1}=(c,x,x-t,x,y,x)$
{\scriptsize $$L_{14}=L\left( \begin{split}
-t &<m\leq 0\\
&k\leq m
\end{split}
\right);\ \ L_{15}=L\left( \begin{split}
-t &<m\leq 0\\
&m<k
\end{split}\right)$$}

\item[(vii)] Subquotients of the singular module $V(T(\bar{v}))$ with$\bar{v}=(c,x,x-t,x,x,x)$:
{\tiny $$L_{16}=L\left( \begin{split}
&-t <n\leq 0\\
&m\leq -t\\
&m< k\leq n
\end{split}\right);\ \ L_{18}:=L\left(\begin{cases}
\ \ \ m\leq n\\
-t<m\leq 0\\
\ \ \ 0<n\\
\ \ \ k\leq m
\end{cases}\bigcup
\begin{cases}
\ \ \ m\leq n\\
-t<m\leq 0\\
\ \ \ n\leq 0\\
\ \ \ k\leq n
\end{cases}\bigcup
\begin{cases}
\ \ \ n<m\\
-t<m\leq 0\\
\ \ \ k\leq n
\end{cases}\right)$$}
{\tiny $$ L_{17}:=L\left(\begin{split}
-t<m\leq 0\\
0<n\ \ \ \\
m<k\leq n
\end{split}\right);\ \ L_{19}:=L\left(\begin{cases}
\ \ \ m\leq n\\
-t<m\leq 0\\
\ \ \ n<k
\end{cases}\bigcup
 \begin{cases}
\ \ \ n<m\\
-t<m\leq 0\\
\ \ \ n\leq -t\\
\ \ \ m<k
\end{cases}\bigcup
 \begin{cases}
\ \ \ n<m\\
-t<m\leq 0\\
-t< n\leq 0\\
\ \ \ n<k
\end{cases}\right)$$}
These are highest weight modules (with respect to some Borel subalgebra) with weight multiplicities bounded by $t$. Modules $L_{16}$ and $L_{17}$ have all $1$-dimensional Gelfand-Tsetlin subspaces.
\item[(viii)] Cuspidal module with weight multiplicities $t$, which is   a subquotient of the singular module $V(T(\bar{v}_1))$ with $\bar{v}_{1}=(c,x,x-t,x,x,z)$, 
{\scriptsize $$L_{20}=L\left( \begin{cases}
\ \ \ m\leq n\\
-t<m\leq 0
\end{cases}\bigcup
\begin{cases}
\ \ \ n<m\\
-t<n\leq 0
\end{cases}\right)$$}
\end{itemize}
\end{Th}

\begin{proof}
Associated with the $\mathfrak{sl}_{3}$-weight $(\lambda_{1}-\frac{2ap}{q},\lambda_{2})$ we have the $\mathfrak{gl}_3$-weight  $$\left(\frac{\lambda_{2}+2\lambda_{1}-\frac{2ap}{q}}{3},\frac{\lambda_{2}-\lambda_{1}+\frac{ap}{q}}{3},\frac{-2\lambda_{2}-\lambda_{1}+\frac{ap}{q}}{3}\right).$$ 
The module $L(\lambda_{1}-\frac{ap}{q},\lambda_{2})$ can be realized as a subquotient of the generic Gelfand-Tsetlin module $V(T(v))$ where     
\begin{center}
 \hspace{1.5cm}\Stone{$c$}\Stone{$x$}\Stone{$x-t$}\\[0.2pt]
 $T(v)$=\hspace{0.3cm} \Stone{$c$}\Stone{$x$}\\[0.2pt]
 \hspace{1.3cm}\Stone{$c$}\\
\end{center}
A basis for this module is given by the set of tableaux $L_{1}$.  Using localization functor we identify another three simple subquotients of $V(T(v))$ with the same annihilator: $L_2$, $L_3$ and $L_4$. Now applying twisted localization functor with respect to $E_{21}$, we obtain two simple subquotients of the generic module $V(T(v_{1}))$ where  
\begin{center}
 \hspace{1.5cm}\Stone{$c$}\Stone{$x$}\Stone{$x-t$}\\[0.2pt]
 $T(v_{1})$=\hspace{0.2cm} \Stone{$c$}\Stone{$x$}\\[0.2pt]
 \hspace{1.3cm}\Stone{$z$}\\
\end{center}
The corresponding basis are $L_5$ and $L_6$. 
Using twisted localization functor with respect to $E_{32}$ we obtain a cuspidal module with all $t$-dimensional weight multiplicities. This module can be realized as a subquotient of the generic Gelfand-Tsetlin module $V(T(v_{2}))$, where
\begin{center}

 \hspace{1.5cm}\Stone{$c$}\Stone{$x$}\Stone{$x-t$}\\[0.2pt]
 $T(v_{2})$=\hspace{0.2cm} \Stone{$y$}\Stone{$x$}\\[0.2pt]
 \hspace{1.3cm}\Stone{$z$}\\
\end{center}

\noindent A basis is given by $L_7$. Now, applying twisted localization functor with respect to $E_{31}$ we obtain two simple subquotients of the generic module $V(T(v_{3}))$, where 
\begin{center}
 \hspace{1.5cm}\Stone{$c$}\Stone{$x$}\Stone{$x-t$}\\[0.2pt]
 $T(v_{3})$=\hspace{0.2cm} \Stone{$z$}\Stone{$x$}\\[0.2pt]
 \hspace{1.3cm}\Stone{$z$}\\
\end{center}

Their bases are $L_8$ and $L_9$. \\
Associated with the $\mathfrak{sl}_{3}$-weight $(\lambda_{2},\frac{ap}{q}-\lambda_{1}-\lambda_{2}-3)$ we have the $\mathfrak{gl}(3)$-weight  $$\left(\frac{\lambda_{2}-\lambda_{1}+\frac{ap}{q}}{3},\frac{\lambda_{2}+2\lambda_{1}-\frac{2ap}{q}}{3},\frac{-2\lambda_{2}-\lambda_{1}+\frac{ap}{q}}{3}\right).$$

\noindent So, the highest weight module  $L(\lambda_{2},\frac{ap}{q}-\lambda_{1}-\lambda_{2}-3)$ has a tableaux realization as a subquotient of the generic Gelfand-Tsetlin module $V(T(u))$ where     
\begin{center}
 \hspace{1.5cm}\Stone{$x$}\Stone{$c$}\Stone{$x-t$}\\[0.2pt]
 $T(u)$=\hspace{0.3cm} \Stone{$x$}\Stone{$c$}\\[0.2pt]
 \hspace{1.3cm}\Stone{$x$}\\
\end{center}
A basis for this module is given by the set of tableaux $L_{10}$.
 Also, the module $V(T(u))$ contains three other simple subquotiens with the desired annihilator. Their bases are $L_{11}$, $L_{12}$, $L_{13}$.

 Applying twisted localization functor with respect to $E_{32}$ we  obtain modules that can be realized as a subquotient of the generic Gelfand-Tsetlin module $V(T(u_{1}))$, where 

\begin{center}

 \hspace{1.5cm}\Stone{$x$}\Stone{$c$}\Stone{$x-t$}\\[0.2pt]
 $T(u_{1})$=\hspace{0.3cm} \Stone{$x$}\Stone{$y$}\\[0.2pt]
 \hspace{1.3cm}\Stone{$x$}\\
\end{center}
These two modules have explicit basis $L_{14}$ and $L_{15}$ respectively.

Finally, 
associated with the $\mathfrak{sl}_{3}$-weight $(\frac{ap}{q}-\lambda_{1}-2,\lambda_{1}+\lambda_{2}-\frac{ap}{q}+1)$ is the $\mathfrak{gl}(3)$-weight  $$\left(\frac{\lambda_{2}-\lambda_{1}+\frac{ap}{q}}{3},\frac{-2\lambda_{2}-\lambda_{1}+\frac{ap}{q}}{3},\frac{\lambda_{2}+2\lambda_{1}-\frac{2ap}{q}}{3}\right).$$ So, the module $L\left(\frac{ap}{q}-\lambda_{1}-2,\lambda_{1}+\lambda_{2}-\frac{ap}{q}+1\right)$ can be realized as a subquotient $L_{16}$ of the singular Gelfand-Tsetlin module $V(T(\bar{v}))$ where     
\begin{center}
 \hspace{1.5cm}\Stone{$x$}\Stone{$x-t$}\Stone{$c$}\\[0.2pt]
 $T(\bar{v})$=\hspace{0.3cm} \Stone{$x$}\Stone{$x$}\\[0.2pt]
 \hspace{1.3cm}\Stone{$x$}\\
\end{center}
The module $V(T(\bar{v}))$ has another three simple subquotients with the same annihilator. Their bases are $L_{17}$, $L_{18}$ and $L_{19}$ respectively.

Applying the twisted localization functor with respect to $E_{21}$ we obtain a cuspidal module with $t$-dimensional weight multiplicities that can be realized as a subquotient of the singular Gelfand-Tsetlin module $V(T(\bar{v}_{1}))$, where\\

\begin{center}

 \hspace{1.5cm}\Stone{$x$}\Stone{$x-t$}\Stone{$c$}\\[0.2pt]
 $T(\bar{v}_{1})$=\hspace{0.3cm}   \Stone{$x$}\Stone{$x$}\\[0.2pt]
 \hspace{1.3cm}\Stone{$z$}\\
\end{center}
A basis for this module is given by $L_{20}$.

\end{proof}

\begin{Rem}
A complete list of non-isomorphic simple admissible Gelfand-Tsetlin modules in the minimal orbit is obtained when $q=2$ and $a=1$ in the theorem above (see Proposition \ref{Pro: Cases for sl_3}).  All other modules correspond to minimal nilpotent orbits in the closure of the principal orbit with $q\geq 3$.

\end{Rem}

\subsection{Principal orbit ($q\geq 3$).} 
By Proposition \ref{Pro: Cases for sl_3}
the elements of $[\overline{Pr}_k^{\mathbb{O}_{prin}}]$ are represented by
 $\mathfrak{sl}_{3}$-weights of the form 
 \begin{align}
\lambda'=\lambda-\frac{p}{q}(\mu+\rho)=(\lambda_{1}-\frac{p}{q}(\mu_{1}+1),\lambda_{2}-\frac{p}{q}(\mu_{2}+1)),
\end{align}
where $\lambda_{1},\lambda_{2},\mu_{1},\mu_{2}\in\mathbb{Z}_{\geq 0}$ are such that $\lambda_{1}+\lambda_{2}\leq p-3$, $\mu_{1}+\mu_{2}\leq q-3$. 
Note that this set is non-empty if and only if $q\geq 3$.
 A $\mathfrak{gl}(3)$ weight associated with $\lambda'$ is $(a_{1},a_{2},a_{3})$, where $a_{1}=\frac{1}{3}\left(2\lambda_{1}+\lambda_{2}-\frac{p}{q}(2\mu_{1}+\mu_{2}+3)\right)$, $a_{2}=\frac{1}{3}\left(\lambda_{2}-\lambda_{1}-\frac{p}{q}(\mu_{2}-\mu_{1})\right)$ and $a_{3}=-\frac{1}{3}\left(\lambda_{1}+2\lambda_{2}-\frac{p}{q}(\mu_{1}+2\mu_{2}+3)\right)$. First we identify simple highest weight modules in the principal orbit, i.e. having the same annihilator as $L(\lambda')$. 
They are all of the form $L(w\circ \lambda')$ with $w$ an element of the Weyl group.
All those highest weight modules are simple Verma modules associated with some $\mathfrak{gl}_3$-weight of the form $(a,b,c)$ satisfying $\{a-b,a-c,b-c\}\cap\mathbb{Z}=\emptyset$. Hence, it is sufficient to describe simple Gelfand-Tsetlin modules in the same orbit with simple generic Verma module which is a  subquotient of  $V(T(v))$, where $v=(x,y,z,x,y,x)$ and $\{x-y,x-z,y-z\}\cap\mathbb{Z}=\emptyset$.

Note that in this orbit we will have both generic and singular Gelfand-Tsetlin modules. 
\begin{Th}\label{thm-q>3}
The following modules constitute a complete list of non-isomorphic infinite-dimensional
simple admissible Gelfand-Tsetlin $\mathfrak{sl}_{3}$-modules 
of level $k$
with denominator $q\geq 3$. 
\begin{itemize}
\item[(i)] Simple admissible modules described in Theorem \ref{Th: Modules associated with minimal orbit} when $q\geq 3$.
\item[(ii)]
Simple generic admissible Gelfand-Tsetlin modules. Any such module  is isomorphic to a simple subquotient of $V(T(v'))$ where $T(v')$ is a strongly generic tableau with top row $(x,y,z):=(a_{1},a_{2}-1,a_{3}-2)$ (see Theorem \ref{characterization of irreducible basis} and Corollary \ref{cor-generic}).
\item[(iii)] Simple singular admissible Gelfand-Tsetlin modules which are isomorphic to one of the following modules:
\begin{itemize}
\item[$\bullet$] Subquotients of $V(T(\bar{v}))$, where $\bar{v}=(x,y,z,x,x,x)$:

$$L_{1}=L\left(\begin{cases}
m\leq 0\\
0<n\\
k\leq m
\end{cases}\bigcup
\begin{cases}
m\leq 0\\
n\leq 0\\
k\leq n
\end{cases}\right); L_{2}=L\left(\begin{cases}
0<m\\
0<n\\
n<k
\end{cases}\bigcup
 \begin{cases}
0<m\\
n\leq 0\\
m<k
\end{cases}\right)$$

$$L_{3}=L\left(\begin{cases}
m\leq 0\\
n<k
\end{cases}\right);\ \ L_{4}=L\left(\begin{cases}
0<m\\
k\leq n
\end{cases}\right);\ \ L_{5}=L\left(\begin{split}
m\leq 0<n\\
m<k\leq n
\end{split}\right)$$

Here, $L_{1},\ L_{2},\ L_{3},\ L_{4}$ are $\mathfrak{sl}_2$-induced modules with unbounded weight multiplicities and $L_{5}$ is cuspidal with infinite weight multiplicities. 
 Modules $L_{1}$ and $L_{2}$ are  induced from $\mathfrak{sl}_2$ (with respect to opposite parabolic subalgebras  where $\mathfrak{sl}_2$ is generated by $E_{23}$, $E_{32}$), while $L_{3}$ and $L_{4}$ 
are  induced from $\mathfrak{sl}_2$ (with respect to opposite parabolic subalgebras  where $\mathfrak{sl}_2$ is generated by $E_{13}$, $E_{31}$).
\item[$\bullet$] Cuspidal modules with infinite weight multiplicities which are subquotients of $V(T(\bar{v}_{1}))$ with $\bar{v}_{1}=(x,y,z,x,x,a)$:
{\scriptsize $$L_{6}=L\left(\begin{cases}
m\leq n\\
m\leq 0
\end{cases}\bigcup
\begin{cases}
n<m\\
m\leq 0
\end{cases}\right);\ \ L_{7}=L\left(\begin{cases}
m\leq n\\
0<m
\end{cases} \bigcup
\begin{cases}
n<m\\
0<m
\end{cases}\right)$$}
\item[$\bullet$] Cuspidal modules with infinite  weight multiplicities which are subquotients of $V(T(\bar{v}_{2}))$ with $\bar{v}_{2}=(x,y,z,a,a,a)$:
{\scriptsize $$L_{8}=L\left(\begin{cases}
m\leq n\\
k\leq n
\end{cases}\bigcup \begin{cases}
n<m\\
k\leq n
\end{cases}\right);\ \ L_{9}=L\left(\begin{cases}
m\leq n\\
n<k
\end{cases}\bigcup
\begin{cases}
n<m\\
n<k
\end{cases}\right)$$}
\item[$\bullet$] Cuspidal module with infinite  weight multiplicities $V(T(\bar{v}_{3}))$ with $\bar{v}_{3}=(x,y,z,a,a,c)$.
\end{itemize}
\end{itemize}
\end{Th}

\begin{proof} The proof is based on the classification and explicit presentation of all simple Gelfand-Tsetlin $\mathfrak{sl}_3$-modules (see \cite{FGR} \S 7 and \S 8).
\begin{itemize}
\item[(i)] Follows by Proposition \ref{Pro: Cases for sl_3}.
\item[(ii)]  These modules are described in Corollary \ref{cor-generic}. 
\item[(iii)] Set $a,c\in\mathbb{C}$ such that $\{x-a,y-a,z-a,c-a\}\cap\mathbb{Z}=\emptyset$. Applying the twisted localization functor with respect to $E_{32}$ with parameter $a$ to the Verma module generated by the tableau 

\begin{center}

 \hspace{1.5cm}\Stone{$x$}\Stone{$y$}\Stone{$z$}\\[0.2pt]
 \hspace{1.2cm}   \Stone{$x$}\Stone{$y$}\\[0.2pt]
 \hspace{1.3cm}\Stone{$x$}\\
\end{center}

we obtain the singular module $V(T(\bar{v}_{0}))$ 
where  

\begin{center}

 \hspace{1.5cm}\Stone{$x$}\Stone{$y$}\Stone{$z$}\\[0.2pt]
 $T(\bar{v}_{0})$=\hspace{0.3cm}   \Stone{$x$}\Stone{$x$}\\[0.2pt]
 \hspace{1.3cm}\Stone{$x$}\\
\end{center}
The  module $V(T(\bar{v}_{0}))$ has five simple subquotiens with the same annihilator. They correspond to the modules $L_{1},L_{2},L_{3}$ and $L_{4}$ (with unbounded but finite dimensional weight multiplicities) and a module with infinite dimensional weight spaces $L_{5}$. Applying twisted localization functor with respect to $E_{21}$ to the previous case, we obtain modules on the singular block $V(T(\bar{v}_{1}))$ where:

\begin{center}

 \hspace{1.5cm}\Stone{$x$}\Stone{$y$}\Stone{$z$}\\[0.2pt]
 $T(\bar{v}_{1})$=\hspace{0.2cm}   \Stone{$x$}\Stone{$x$}\\[0.2pt]
 \hspace{1.3cm}\Stone{$a$}\\
\end{center}
The module $V(T(\bar{v}_{1}))$ has two simple subquotiens with the same annihilator, both of them having infinite weight multiplicities. They correspond to the modules $L_{6}$ and $L_{7}$. Consider\\

\begin{center}

 \hspace{1.5cm}\Stone{$x$}\Stone{$y$}\Stone{$z$}\\[0.2pt]
 $T(\bar{v}_{2})$=\hspace{0.2cm}   \Stone{$a$}\Stone{$a$}\\[0.2pt]
 \hspace{1.3cm}\Stone{$a$}\\
\end{center}

\noindent The singular module $V(T(\bar{v}_{2}))$ has two simple subquotiens with infinite dimensional weight spaces corresponding to  $L_{8}$ and $L_{9}$. Finally, consider\\

\begin{center}

 \hspace{1.5cm}\Stone{$x$}\Stone{$y$}\Stone{$z$}\\[0.2pt]
 $T(\bar{v}_{3})$=\hspace{0.2cm}   \Stone{$a$}\Stone{$a$}\\[0.2pt]
 \hspace{1.3cm}\Stone{$c$}\\
\end{center}

which is obtained by application on the twisted localization functor with respect to $E_{21}$ to the previous case
\noindent The singular module $V(T(\bar{v}_{3}))$ is simple with infinite dimensional weight multiplicities. It  corresponds to $L_{10}$.
\end{itemize}

\end{proof}

\begin{Rem}
Modules $L_1$, $\ldots$, $L_4$ in Theorem \ref{thm-q>3} have unbounded finite weight multiplicities which grow linearly. They are isomorphic to the corresponding generalized Verma modules (induced from unbounded simple $sl(2)$-modules).
\end{Rem}

\section{Admissible $\mathfrak{sl}_n$-modules induced from $\mathfrak{sl}_k$-modules}
Let 
$k$ be an admissible number for $\widehat{\mf{sl}}_n$,
so that 
$$k+n=p/q,\quad  p\geq n,\ q\geq 1,\ (p,q)=1.$$
In this section we describe families of simple admissible $\fing=\mathfrak{sl}_n$-modules induced from simple $\mathfrak{sl}_k$-modules for $k=2,3$. 
If  $M=L_{\mathfrak p}(\lambda, N)$ is admissible for some parabolic subalgebra $\mf{p}$ with Levi subalgebra $\mf{l}$
and nilradical $\mf{m}$,
then   
$M^{\mathfrak{m}}\simeq N$ and, hence, $N$ is  admissible $\mathfrak l$-module
by Theorem \ref{Th:restriction}. Therefore, we need to consider $\mathfrak{sl}_n$-modules $M$ induced from simple admissible $\mathfrak{sl}_k$-modules. 

We will provide a tableaux realization for admissible $\mathfrak{sl}_n$-modules induced from simple Gelfand-Tsetlin  $\mathfrak{sl}_k$-modules in the principal orbit for $k=2,3$.  

\begin{Rem}
We defined Gelfand-Tsetlin $\mathfrak{gl}_n$-modules in connection with the fixed Gelfand-Tsetlin subalgebra (which is maximal commutative) determined by a special chain of embeddings of $\mathfrak{gl}_i$'s into each other (left upper corner inclusion). On the other hand, one can consider a different chain of embeddings starting from an arbitrary root of $\mathfrak{gl}_n$. This would produce a different maximal commutative subalgebra of $U(\mathfrak{gl}_n)$ and one could consider  Gelfand-Tsetlin  modules with respect to this subalgebra. But, it is clear, that all such chains of embeddings are conjugated by the Weyl group and, hence, the categories of corresponding Gelfand-Tsetlin modules are equivalent.     Same holds for Gelfand-Tsetlin modules over  
$\mathfrak{sl}_n$.
If  $\mathfrak{sl}_k$  is a root subalgebra of $\mathfrak{sl}_n$ then there exists a parabolic subalgebra $\mathfrak p=
\mathfrak l\oplus \mathfrak n$ with $\mathfrak l=\mathfrak{sl}_k+\h$. Then for a Gelfand-Tsetlin $\mathfrak{sl}_k$-module $N$ with respect to certain choice of embeddings of $\mathfrak{sl}_k$, the induced module $M_{\mathfrak p}(\lambda, N)$  will be 
 a 
 Gelfand-Tsetlin module for a certain chain of embeddings of $\mathfrak{sl}_n$ which contains  $\mathfrak{sl}_k$.   This module in general may not be  a 
 Gelfand-Tsetlin module with respect to the Gelfand-Tsetlin subalgebra (left upper corner inclusion).  But, due to the category equivalence mentioned above, we can always assume
  without loss of generality that $\mathfrak{sl}_k$ subalgebra belongs to the chain of the left upper corner inclusions and that $M_{\mathfrak p}(\lambda, N)$  is a Gelfand-Tsetlin module with respect to the Gelfand-Tsetlin subalgebra.

\end{Rem}

\subsection{Construction of induced modules}

Let $N$ be any simple $\mathfrak{sl}_k$-module ($k=2,3$) with tableaux realization as a subquotient of a module $V(T(u))$ with generic top row (i.e. $u_{21}-u_{22}\notin\mathbb{Z}$ for $k=2$ and $\{u_{31}-u_{32},\ u_{31}-u_{33},\ u_{32}-u_{33}\} \cap \mathbb{Z}=\emptyset$ for $k=3$). 

\textbf{$\mathfrak{sl}_{2}$-case.} Set $v_{1}:=u_{21}$ and $v_{2}:=u_{22}$. Choose  also $\{v_{i}\}_{i=3,\ldots,n}$ complex numbers such that $v_{i}-v_{j}\notin\mathbb{Z}$ for any $1\leq i< j\leq n$. Let $T(v)$ be the Gelfand-Tsetlin tableau with entries
\begin{align}
v_{ij}=\begin{cases}
v_{j}, & \ \ \text{if}\ \  i\geq 2\\
u_{11}, & \ \ \text{if}\ \  i=1
\end{cases}
\end{align}

for $1\leq j\leq i\leq n$.

 \begin{Lem}\label{Lem: Induced by sl(2)}
If $T(v)$ is the above Gelfand-Tsetlin tableau  then  there exist a subquotient $M$ of $V(T(v))$ isomorphic to an  induced module $M_{\mathfrak p}(\lambda, N)$, where ${\mathfrak p}$ has the Levi subalgebra isomorphic to $\mathfrak{sl}_2$ and $N$  is cuspidal $\mathfrak{sl}_2$-module. 
Moreover, $M$  is simple if and only if $u_{21}-u_{11}, u_{22}-u_{11}\notin\mathbb{Z}$. It has a basis  given by the set of tableaux $T(w)$ in $\mathcal{B}(T(v))$ such that $w_{rs}-w_{r-1,s}\in\mathbb{Z}_{\geq 0}$ for any $ 3\leq r\leq n,\ 1\leq s\leq r-1$. 

\end{Lem}
\begin{proof}
By construction the module $V(T(v))$ is generic. The assertions follow from Theorems \ref{Generic Gelfand-Tsetlin modules} and \ref{characterization of irreducible basis}.
\end{proof}

Denote by $M(N,v_{3},\ldots,v_{n})$ the induced module $M$  described in Lemma \ref{Lem: Induced by sl(2)}.

\textbf{$\mathfrak{sl}_3$-case.} Set $v_{1}:=u_{31}$, $v_{2}:=u_{32}$ and $v_{3}:=u_{33}$. Choose $\{v_{i}\}_{i=4,\ldots,n}$ complex numbers such that $v_{i}-v_{j}\notin\mathbb{Z}$ for any $1\leq i< j\leq n$. Let $T(v)$ be the Gelfand-Tsetlin tableau with entries
\begin{align}
v_{ij}=\begin{cases}
v_{j}, & \ \ \text{if}\ \  i\geq 3\\
u_{ij}, & \ \ \text{if}\ \  i\leq 2
\end{cases}
\end{align}
for $1\leq j\leq i\leq n$.

 \begin{Lem}\label{Lem: Induced by sl(3)}
If $T(v)$ is the above Gelfand-Tsetlin tableau  then  there exists a subquotient $M$ of $V(T(v))$ isomorphic to an  induced module $M_{\mathfrak p}(\lambda, N)$, where ${\mathfrak p}$ has the Levi subalgebra isomorphic to $\mathfrak{sl}_3$ and $N$  is cuspidal $\mathfrak{sl}_3$-module.  If $v$ is generic then $M$ has a basis  given by the set of tableaux $T(w)$ in $\mathcal{B}(T(v))$ such that $w_{rs}-w_{r-1,s}\in\mathbb{Z}_{\geq 0}$ for any $4\leq r\leq n,\ 1\leq s\leq r-1$. Moreover, $M$  is simple if and only if $u_{rs}-u_{r-1,t}\notin\mathbb{Z}$ for any $1\leq s\leq r\leq 3$ and $1\leq t\leq r-1$. 
\end{Lem}
\begin{proof}
If $v$ is singular then the tableau $T(v)$ and the module $V(T(v))$ is $1$-singular (see \cite{FGR16} and Remark \ref{Rem: 1-singular modules}). If  $v$ is generic then the module $V(T(v))$ is generic and the assertions follows from Theorems \ref{Generic Gelfand-Tsetlin modules} and \ref{characterization of irreducible basis}.
\end{proof}

Denote  by $M(N,v_{4},\ldots,v_{n})$ the induced module described in Lemma \ref{Lem: Induced by sl(3)}.

\begin{Rem}
\begin{itemize} Set $N$ a simple $\mathfrak{sl}_{3}$-module.
\item[(i)] If $N$ is not cuspidal then  $M(N,v_{4},\ldots,v_{n})$ is isomorphic to a Verma module or to a Generalized Verma module induced from $\mathfrak{sl}_2$.
\item[(ii)] If $N$ is cuspidal then all weight multiplicities of $M(N,v_{4},\ldots,v_{n})$ are infinite.
\end{itemize}
\end{Rem}

\subsection{Admissible $\mathfrak{sl}_2$-induced modules in the principal orbit revisited.} 
Any simple weight $\mathfrak{sl}_2$-module is a 
Gelfand-Tsetlin module and it admits a tableaux realization. Since we are interested in induced non-highest weight modules we can consider only 
cuspidal $\mathfrak{sl}_2$-modules. Any such module is a module in the principal nilpotent orbit 
of $\mf{sl}_2$
and it is isomorphic to $V(T(u))$ for some generic $u$. In particular, for admissible $\mathfrak{sl}_2$-module  $V(T(u))$  in the principal nilpotent orbit we have $u=(u_{21},u_{22},u_{11})$ where $u_{21}-u_{22}-1=\lambda-\frac{p}{q}(\mu+1)$ for some $\lambda,\mu\in\mathbb{Z}_{\geq 0}$ satisfying $\lambda \leq p-2$ and $\mu \leq q-2$ with $p,q\in \N$, $(p,q)=1$, $p,q\geq 2$. We say in this case that $V(T(u))$ (and its simple quotient which contains $T(u)$) is associated with the weight $\lambda-\frac{p}{q}(\mu+1)$.  Induced module from a generic   $\mathfrak{sl}_2$-module need not to be generic. In fact, there exist singular induced modules as we saw in the case of  $\mathfrak{sl}_3$.

The following theorem gives a complete classification of all simple weight generic admissible $\mathfrak{sl}_2$-induced modules in the principal orbit.  

\begin{Th}\label{Th: Admissible induced by sl(2)}
Let $N$  be a simple admissible cuspidal $\mathfrak{sl}_2$-module associated with the weight $\lambda'=\lambda_{1}-\frac{p}{q}(\mu_{1}+1)$.  If $\lambda_{1}\leq p-n$, $\mu_{1}\leq q-n$ and $p,q\geq n$ then for any choice of $\{\lambda_{i}, \mu_{i}\}_{i=2,\ldots,n-1}\subseteq \mathbb{Z}_{\geq 0}$  such that $\sum_{i=1}^{n-1}\lambda_{i}\leq p-n$ and $\sum_{i=1}^{n-1}\mu_{i}\leq q-n$, there exist complex numbers $\{v_{3},\ldots,v_{n}\}$ such that the induced module $M(N,v_{3},\ldots,v_{n})$ is a simple admissible generic $\mathfrak{sl}_n$-module in the principal orbit. Moreover, these modules exhaust all simple weight admissible 
generic $\mathfrak{sl}_n$-modules in the principal orbit induced from $\mathfrak{sl}_2$.
\end{Th}
\begin{proof} Note that $N$ belongs to the nilpotent orbit and, hence, strongly generic. 
Set $v_{1}=u_{21}$, $v_{2}=u_{22}$ and choose $v_{3},\ldots,v_{n}$ such that $v_{i}-v_{i+1}-1=\lambda_{i}-\frac{p}{q}(\mu_{i}+1)$ for any $i=1,\ldots,n-1$. The statements follow from Lemma \ref{Lem: Induced by sl(2)}.
\end{proof}
\begin{Rem}
Theorem \ref{Th: Admissible induced by sl(2)} is a ratification of Corollary \ref{cor-generic} for generic modules in the principal orbit. In order to complete a classification of all simple admissible $\mathfrak{sl}_2$-induced modules one needs to consider non-principal orbits and singular modules. In the case of 
$\mathfrak{sl}_{3}$ such modules were constructed in  Theorem \ref{Th: Modules associated with minimal orbit} and Theorem \ref{thm-q>3}. 
\end{Rem}

\subsection{Admissible $\mathfrak{sl}_3$-induced modules}
 
Any simple admissible $\mathfrak{sl}_3$-module in the principal  nilpotent orbit admits a tableaux realization as a subquotient of some Gelfand-Tsetlin module $V(T(u))$ where $u=(u_{31},u_{32},u_{33},u_{21},u_{22},u_{11})$ and 
  for $i=1,2$,  $u_{3i}-u_{3,i+1}-1=\lambda_{i}-\frac{p}{q}(\mu_{i}+1)$ for some $\lambda_{i},\mu_{i}\in\mathbb{Z}_{\geq 0}$ satisfying $\lambda_{1}+\lambda_{2}\leq p-3$ and $\mu_{1}+\mu_{2}\leq q-3$ with $p,q\in \N$, $(p,q)=1$, $p,q\geq 3$.  We say in this case that $V(T(u))$ (and its simple quotient which contains $T(u)$)  is associated with the weight $(\lambda_{1}-\frac{p}{q}(\mu_{1}+1),\lambda_{2}-\frac{p}{q}(\mu_{2}+1))$.

\begin{Th}\label{Th: Admissible induced by sl(3)}
Let $N$ to be a simple admissible $\mathfrak{sl}_3$-module in the principal nilpotent orbit associated with the weight $\lambda'=(\lambda_{1}-\frac{p}{q}(\mu_{1}+1),\lambda_{2}-\frac{p}{q}(\mu_{2}+1))$ such that $\lambda_{1}+\lambda_{2}\leq p-n$ and $\mu_{1}+\mu_{2}\leq q-n$ and $p,q\geq n$. Let $\{\lambda_{i}, \mu_{i}\}_{i=3,\ldots,n-1}\subseteq \mathbb{Z}_{\geq 0}$  such that $\sum_{i=1}^{n-1}\lambda_{i}\leq p-n$ and $\sum_{i=1}^{n-1}\mu_{i}\leq q-n$. Set $v_{1}=u_{31}$, $v_{2}=u_{32}$, $v_{3}=u_{33}$ and choose complex numbers $\{v_{4},\ldots,v_{n}\}$ such that $v_{i}-v_{i+1}-1=\lambda_{i}-\frac{p}{q}(\mu_{i}+1)$ for any $i=1,\ldots,n-1$.
Then
the induced module $M(N,v_{4},\ldots,v_{n})$ is an admissible $\mathfrak{sl}_n$-module in the principal orbit. Moreover, simple modules $M(N,v_{4},\ldots,v_{n})$ exhaust all simple admissible  $\mathfrak{sl}_3$-induced modules in the principal nilpotent orbit. 
\end{Th}

\begin{proof}
The statement is clear from the construction.
\end{proof}

  Theorem \ref{Th: Admissible induced by sl(3)}   together with Lemma \ref{Lem: Induced by sl(3)} and Theorem \ref{thm-q>3}    
  give a complete classification of all simple weight admissible $\mathfrak{sl}_3$-induced modules $M(N,v_{4},\ldots,v_{n})$  in the principal nilpotent orbit. 
  
 \begin{Rem}
 Theorem \ref{Th: Admissible induced by sl(3)} describes both generic and singular $\mathfrak{sl}_3$-induced modules $M(N,v_{4},\ldots,v_{n})$.  All these modules have infinite weight multiplicities. Generic modules were described already 
 in Corollary \ref{cor-generic}. 
 \end{Rem}

\bibliographystyle{alpha}


\end{document}